\documentclass[a4paper]{article}
\usepackage[utf8]{inputenc}
\usepackage{amsthm,amsmath,amssymb,amsfonts}
\usepackage{tikz}
\usepackage{palatino}

\usepackage{hyperref}
\hypersetup{
    colorlinks=true,
    linkcolor=black,
    urlcolor=blue,
    citecolor=black,
    pdftitle={Collapsing Constructive and Intuitionistic Modal Logics},
    pdfauthor={Leonardo Pacheco}
}

\theoremstyle{plain}
\newtheorem{theorem}{Theorem}
\newtheorem{lemma}[theorem]{Lemma}
\newtheorem{proposition}[theorem]{Proposition}

\theoremstyle{definition}
\newtheorem{definition}{Definition}

\newcommand{\ck}{\mathsf{CK}}
\newcommand{\isf}{\mathsf{IS5}}
\newcommand{\muisf}{\mathsf{\mu IS5}}

\newcommand{\tuple}[1]{{\langle #1 \rangle}}
\newcommand{\verifier}{\mathsf{V}}
\newcommand{\refuter}{\mathsf{R}}
\newcommand{\role}{\mathsf{Q}}
\newcommand{\dualrole}{\bar{\mathsf{Q}}}

\newcommand{\sig}[1]{{\mathrm{sig}^\mathsf{I}\langle #1 \rangle}}
\newcommand{\sub}[1]{{\mathrm{Sub}(#1)}}

\title{Game semantics for the constructive $\mu$-calculus}
\author{Leonardo Pacheco \\
{\small Institute of Discrete Mathematics and Geometry, TU Wien, Austria.} \\
{\small \texttt{leonardovpacheco@gmail.com}}}
\date{}

\begin{document}
\maketitle

\begin{abstract}
    We define game semantics for the constructive $\mu$-calculus and prove its equivalence to bi-relational semantics.
    As an application, we use the game semantics to prove that the $\mu$-calculus collapses to modal logic over the modal logic $\isf$.
    We then show the completeness of $\isf$ extended with fixed-point operators.
\end{abstract}

\section{Introduction}

This paper is a first step into relating two strands of research modal logic: the modal $\mu$-calculus and constructive modal logics.
We define a constructive variant of the $\mu$-calculus by adding least and greatest fixed-point operators to constructive modal logic.
We define game semantics for the constructive $\mu$-calculus and prove its equivalence to bi-relational Kripke semantics.
We use then the game semantics to study an intuitionistic variant of the modal logic $\mathsf{S5}$ with fixed-point operators.
Before introducing our results, we briefly review the related literature on constructive modal logics and the $\mu$-calculus.

On constructive modal logics, the duality of the modalities $\Box$ and $\Diamond$ is lost.
These logics have been studied for a long time; some of the first texts on the topic are Fitch \cite{fitch1948intuitionistic} and Prawitz \cite{prawitz1965natural}.
In this paper, we use Mendler and de Paiva's $\ck$ bi-relational $\ck$-models \cite{mendler2005constructive}.
These models are based on those of Wijesekera \cite{wijesekera1990constructive}, but allow worlds where the false proposition $\bot$ holds.

The $\ck$-models are not the only semantics available for constructive modal logics.
Of note the semantics of Acclavio \emph{et al.} \cite{acclavio2021game}.
Acclavio \emph{et al.} provide complete denotational semantics for $\ck$ via game semantics.
Their games are canonical representations of proofs, not model checking games as the ones presented in this paper.
There are also categorical semantics \cite{alechina2001categorical} and realizability semantics \cite{kuznets2021justification}.

Furthermore, one should note that constructive modal logic is not the only non-classical variant of modal logic.
It can also be strengthened to intuitionistic and G\"odel modal logics.
On the axiomatic side, these logics are obtained by adding axioms to constructive modal logic.
On the semantics side, they are obtained by excluding fallible worlds and adding restrictions on $\ck$-models.
See \cite{das2023diamonds,degroot2024semantical} for more information on the relation between constructive and intuitionistic modal logic.

While models for intuitionistic modal logics can be seen as a particular type of $\ck$-models, constructive and intuitionistic variants of the same logic usually behave quite differently.
Of note is Das and Marin's \cite{das2023diamonds} paper which shows that the $\Diamond$-free fragment of $\ck$ and $\mathsf{IK}$ do not coincide: $\ck$ does not prove $\neg\neg\Box\bot \to \Box\bot$, while $\mathsf{IK}$ does.

The modal $\mu$-calculus was defined by Kozen \cite{kozen1983results}, who also defined a related proof system $\mathsf{\mu K}$.
The completeness of $\mathsf{\mu K}$ was first proved by Walukiewicz \cite{walukiewicz1995completeness}.
See \cite{lenzi2010recent,bradfield2018mucalculus} for surveys on the $\mu$-calculus.

The $\mu$-calculus' alternation hierarchy classifies the $\mu$-formulas by how many alternating least and greatest fixed-point operators they contain.
The strictness of the hierarchy was open for many years until it was proved by Bradfield \cite{bradfield1998strict}.
Bradfield later gave a simplified proof of the alternation hierarchy's strictness using  evaluation games \cite{bradfield1998simplifying}.
The strictness may not hold over restricted classes of models.
For example, Alberucci and Facchini \cite{alberucci2009modal} proved that the alternation hierarchy collapses to its alternation-free fragment over transitive models, and to modal logic over equivalence relations.
See Chapter 2 of \cite{pacheco2023exploring} for a survey on the alternation hierarchy.

The $\mu$-formulas are famously hard to understand.
One advantage of game semantics for the $\mu$-calculus over the standard Kripke semantics is that they give a more intuitive interpretation of the $\mu$-formulas.
Furthermore, evaluation games are also useful as a tool for proving theorems about the $\mu$-calculus.

In an evaluation game for the $\mu$-calculus, two players discuss whether a formula is true at a given world of a Kripke model.
In the classical version of the game, it is usual to refer to the players as Verifier and Refuter.
In the constructive version of the game, we will still have two players, but now they alternate between the roles of Verifier and Refuter, depending on their moves.
This difference happens because, over classical semantics, every formulas can be put in negative normal form; this allows us to simplify the evaluation games in the classical case.
In other words, we need to consider negation and implication in constructive semantics.
Therefore we will need a more delicate argument to prove the equivalence of the semantics in the constructive case.
Our proof is based on the proof of the correctness of game semantics for the classical $\mu$-calculus by Ong \cite{ong2015automata}.

Since our evaluation games build on the $\ck$-models of Mendler and de Paiva \cite{mendler2005constructive}, the game semantics can also be used for any logic whose semantics are based on (subsets of) $\ck$-models.
In particular, our game semantics can also be used to define semantics for an intuitionistic $\mu$-calculus, based on bi-relational models for the modal logic $\mathsf{IK}$.

As an application, we study the logic $\isf$, an intuitionistic variant of $\mathsf{S5}$ with fixed-points operators.
$\isf$ is also known as $\mathsf{MIPQ}$ and $\mathsf{MIPC}$, and was first studied by Prior \cite{prior1957time}.
The completeness of $\isf$ over $\isf$-models was proved by Ono \cite{ono1977intuitionistic} and Fischer Servi \cite{fischerservi1978finite}.
We use the game semantics to show that the constructive $\mu$-calculus collapses to constructive modal logic over $\isf$.
That is, every $\mu$-formula is equivalent to a formula without fixed-point operators over $\isf$-models.
Our proof is a generalization of Alberucci and Facchini's proof of the collapse of the (classical) $\mu$-calculus to (classical) modal logic over $\mathsf{S5}$-models \cite{alberucci2009modal}.
Finally, we use the $\mu$-calculus' collapse to modal logic over $\isf$ to prove the completeness of $\muisf$, the modal logic obtained by adding fixed-point axioms and rules to the modal logics $\isf$.
As far as the author is aware, these are the first completeness results for any logic over the constructive $\mu$-calculus.

At last, we note that a constructive variant $\mathsf{CS5}$ of $\mathsf{S5}$ was previously studied by Arisaka \emph{et al.} \cite{arisaka2015nested}, who defined and proved the correctness of a nested sequent calculus for $\mathsf{CS5}$.
The bi-relational semantics for this logic has not been studied yet in the literature, so the semantical methods we use to prove the collapse of the $\mu$-calculus over $\isf$ cannot be used for $\mathsf{CS5}$.

\paragraph{Outline}
In Section \ref{sec::preliminaries}, we define the syntax and bi-relational Kripke semantics for the constructive $\mu$-calculus, and review the modal logic $\isf$.
In Section \ref{sec::game-semantics}, we define the game semantics for the constructive $\mu$-calculus and prove its equivalence to Kripke semantics.
In Section \ref{sec::collapse}, we prove the constructive $\mu$-calculus' collapse to modal logic over $\isf$-models.
In Section \ref{sec::mu-is5}, we prove the completeness of $\muisf$.
In Section \ref{sec::future-work}, we point some directions for future work.

\paragraph{Acknowledgements}
I would like to thank David Fernández-Duque, Iris van der Giessen, and  Konstantinos Papafilippou for the discussions we had about constructive modal logics and the $\mu$-calculus.
I would also like to thank Thibaut Kouptchinsky for comments on an early draft and the anonymous reviewers for their comments.
The comments from the anonymous reviewers greatly improved this paper.
This research was partially funded by the FWF grant TAI-797.

\section{Constructive \texorpdfstring{$\mu$}{mu}-calculus}
\label{sec::preliminaries}
\paragraph{Syntax}
The language of the $\mu$-calculus is obtained by adding \emph{least} and \emph{greatest fixed-point operators} $\mu$ and $\nu$ to the language of modal logic.
When defining the fixed-point formulas $\mu X.\varphi$ and $\nu X.\varphi$, we have a syntactical requirement in order to have well-defined semantics.

Fix a set $\mathrm{Prop}$ of proposition symbols and a set $\mathrm{Var}$ of variable symbols.
The \emph{constructive $\mu$-formulas} are defined by the following grammar:
\[
    \varphi := P \mid X \mid \bot \mid \top \mid \neg \varphi \mid \varphi\land\varphi  \mid \varphi\lor\varphi \mid \varphi\to\varphi \mid \Box\varphi \mid \Diamond\varphi \mid  \mu X.\varphi \mid \nu X.\varphi,
\]
where $P$ is a proposition symbol, $X$ is a variable symbol.
A $\mu$-formula without fixed-point operators is called a \emph{modal formula}.
The fixed-point formulas $\mu X.\varphi$ and $\nu X.\varphi$ are defined iff $X$ is positive in $\varphi$.
\begin{definition}
    We classify $X$ as \emph{positive} or \emph{negative} in a given formula by structural induction:
    \begin{itemize}
        \item $X$ is positive and negative in $P$ and in $\bot$;
        \item $X$ is positive in $X$ ;
        \item if $Y\neq X$, $X$ is positive and negative in $Y$;
        \item if $X$ is positive (negative) in $\varphi$, then $X$ is negative (positive) in $\neg\varphi$;
        \item if $X$ is positive (negative) in $\varphi$ and $\psi$, then $X$ is positive (negative) in $\varphi\land \psi$, $\varphi\lor \psi$, $\Box\varphi$, and $\Diamond\varphi$;
        \item if $X$ is negative (positive) in $\varphi$ and positive (negative) in $\psi$, then $X$ is positive (negative) in $\varphi\to \psi$;
        \item $X$ is positive and negative in $\mu X.\varphi$ and $\nu X.\varphi$.
    \end{itemize}
\end{definition}

While burdensome, we need to consider the positiveness and negativeness of variables to guarantee that the semantics for fixed-points formulas are well-defined.
This contrasts with the classical $\mu$-calculus, where it is common to suppose every formula is in negative normal form, and so we can set up the grammar in a way only positive variables occur.
We cannot do the same on constructive semantics: for example, $\neg\Diamond\neg\varphi$ is not equivalent to $\Box\varphi$ over constructive semantics.

Denote the set of subformulas of $\varphi$ by $\sub{\varphi}$.
We we use $\eta$ to denote either $\mu$ or $\nu$, and $\triangle$ to denote $\Box$ or $\Diamond$.
An occurrence of a variable $X$ in a formula $\varphi$ is \emph{bound} iff it in the scope of a fixed-point operator $\eta X$.
An occurrence of $X$ is \emph{free} iff it is not bound.
We treat free occurrences of variable symbols as propositional symbols.
A formula $\varphi$ is \emph{closed} iff it has no free variables.
An occurrence of $X$ in $\varphi$ is \emph{guarded} iff it is in the scope of some modality $\triangle$.
A formula $\varphi$ is \emph{guarded} iff, for all $\eta X.\psi\in\sub{\varphi}$, $X$ is guarded in $\psi$.
A formula $\varphi$ is \emph{well-bounded} iff, for all variables $X$ occurring bounded in $\varphi$, $X$ occurs only once and there is only one fixed-point operator $\eta X$ in $\varphi$.
A formula is \emph{well-named} iff it is guarded and well-bounded.
Every formula is equivalent to a well-named formula.
If $\varphi$ is a well-named formula and $\eta X.\psi\in\sub{\psi}$, denote by $\psi_X$ the formula $\psi$ which is bound by the fixed-point operator $\eta X$.

\paragraph{Semantics}
We consider the bi-relational semantics defined by Mendler and de Paiva \cite{mendler2005constructive}.
\begin{definition}
    A $\ck$-model is a tuple $M=\tuple{W, W^\bot ,\preceq, R, V}$ where: $W$ is the set of \emph{possible worlds}; $W^\bot\subseteq W$ is the set of \emph{fallible worlds}; $\preceq$ is a reflexive and transitive relation over $W$; $R$ is a relation over $W$; and $V:\mathrm{Prop}\to \mathcal{P}(W)$ is a valuation function.
    We call $\preceq$ the \emph{intuitionistic relation} and $R$ the \emph{modal relation}.
    We require that, if $w \preceq v$ and $w\in V(P)$, then $v\in V(P)$; and that $W^\bot\subseteq V(P)$, for all $P\in\mathrm{Prop}$.
    We also require that if $w\in W^\bot$, then, if either $w\preceq v$ or $wRv$, then $v\in W^\bot$ too.
    For convenience, we sometimes write $w\succeq v$ iff $v\preceq w$.
\end{definition}

In Section \ref{sec::game-semantics}, we will consider models with augmented valuations when proving the correctness of game semantics.
When augmenting $M$, we treat some variable symbol $X$ as a proposition symbol and assign a value to it.
Formally, let $M = \tuple{W, W^\bot,\preceq, R,V}$ be a $\ck$-model, $A\subseteq W$ and $X$ be a variable symbol; the \emph{augmented $\ck$-model} $M[X\mapsto A]$ is obtained by setting $V(X) := A$.
Given any $\mu$-formula $\varphi$, we also define $\|\varphi(A)\|^M := \|\varphi(X)\|^{M[X\mapsto A]}$.

Fix a $\ck$-model $M = \tuple{W, W^\bot,\preceq, R,V}$.
Given a $\mu$-formula $\varphi$, define the operator $\Gamma_{\varphi(X)}(A) := \|\varphi(A)\|^M$.
Define the valuation of the $\mu$-formulas over $M$ by induction on the structure of the formulas:
\begin{itemize}
    \item $\|P\|^M = V(P)$;
    \item $\|\bot\|^M = W^\bot$;
    \item $\|\top\|^M = W$;
    \item $\|\varphi\land\psi\|^M = \|\psi\|^M \cap \|\psi\|^M$;
    \item $\|\varphi\lor\psi\|^M = \|\psi\|^M \cup \|\psi\|^M$;
    \item $\|\varphi\rightarrow\psi\|^M = \{w\mid \text{for all $v$, if } w\preceq v \text{, then } v\in \|\varphi\|^M \text{ implies } v\in\|\psi\|^M \}$;
    \item $\|\neg\varphi\|^M = \{ w\mid \text{for all $v$, if $w\preceq v$, then } v\not\models\varphi\}$;
    \item $\|\Box\varphi\|^M = \{ w \mid \text{for all $v$ and $w$, if } w \preceq v \text{ and } v R u \text{, then }u\in\|\varphi\|^M\}$;
    \item $\|\Diamond\varphi\|^M = \{ w \mid \text{for all $v$, if $w \preceq v$, then there is $v R u$ such that }u\in\|\varphi\|^M\}$;
    \item $\|\mu X.\varphi(X)\|^M$ is the least fixed-point of the operator $\Gamma_{\varphi(X)}$; and
    \item $\|\nu X.\varphi(X)\|^M$ is the greatest fixed-point of the operator $\Gamma_{\varphi(X)}$.
\end{itemize}
We also write $M,w\models\varphi$ when $w\in\|\varphi\|^M$.
We omit the reference to the model $M$ when it is clear from the context and write $w\in\|\varphi\|$ and $w\models\varphi$.

The Knaster--Tarski Theorem \cite{arnold2001rudiments} states that every monotone operator has least and greatest fixed-points.
The following proposition implies that the valuations of the fixed-point formulas $\mu X.\varphi$ and $\nu X.\varphi$ are well-defined.
\begin{proposition}
    \label{prop::monotoness_of_Gamma_varphi}
    Fix a $\ck$-model $M=\tuple{W, W^\bot,\preceq,R, V}$ and a formula $\varphi(X)$ with $X$ positive.
    Then the operator $\Gamma_{\varphi(X)}$ is monotone.
    Therefore the valuations of the fixed-point formulas $\mu X.\varphi$ and $\nu X.\varphi$ are well-defined.
\end{proposition}
\begin{proof}
    Fix a model $M$ and sets $A\subseteq B \subseteq W$.
    We prove that if $X$ is positive in $\varphi$ then $\|\varphi(A)\|\subseteq \|\varphi(B)\|$; and that if $X$ is negative in $\varphi$, then $\|\varphi(B)\|\subseteq \|\varphi(A)\|$.
    We will focus on the positive case, as the negative case is dual.

    The proof is by structural induction on the $\mu$-formulas.
    The cases of formulas of the form $P$, $X$, $Y$, $\varphi\land\psi$, and $\varphi\lor\psi$ follow by direct calculations.
    The case for formulas of the form $\eta X.\varphi$ is trivial, as $X$ is not free in $\eta X.\varphi$.

    We now prove the proposition for formulas of the form $\varphi\to\psi$.
    Suppose $X$ is positive in $\varphi\to \psi$, then $X$ is positive in $\psi$ and negative in $\varphi$.
    Therefore:
    \begin{align*}
      w\in\|(\varphi\to\psi)(A)\| \iff& \text{for all $v$, if $w \preceq v$ and }v\in \|\varphi(A)\|\text{, then } v\in\|\psi(A)\| \\
      \implies& \text{for all $v$, if $w \preceq v$ and }v\in \|\varphi(B)\|\text{, then } v\in\|\psi(B)\| \\
      \iff&w\in  \|(\varphi\to\psi)(B)\|.
    \end{align*}
    The case for formulas of the form $\neg\varphi$ is similar.

    Finally, we prove the proposition for formulas of the form $\Box\varphi$.
    Suppose $X$ is positive in $\Box\varphi$, then $X$ is positive in $\varphi$.
    Therefore:
    \begin{align*}
        w\in\|\Box\varphi(A)\| \iff& \text{for all $v, u$ such that $w\preceq v R u$, }u\in \|\varphi(A)\| \\
        \implies& \text{for all $v, u$ such that $w\preceq v R u$, }u\in \|\varphi(B)\| \\
        \iff& w\in \|\Box\varphi(B)\|.
    \end{align*}
    The proof for formulas of the form $\Diamond\varphi$ is similar.
\end{proof}

We will also need to consider the \emph{approximants} $\eta X^\alpha.\varphi$ of fixed-point formulas $\eta X.\varphi$, for all ordinal number $\alpha$.
Fix a $\ck$-model $M$ and a formula $\varphi(X)$ where $X$ is positive.
The $\alpha$th approximant $\mu X^\alpha.\varphi$ of $\mu X.\varphi$ is obtained by applying $\Gamma_\varphi(X)$ to $\emptyset$, $\alpha$ many times.
Similarly, the approximant $\nu X^\alpha.\varphi$ is obtained by applying $\Gamma_\varphi(X)$ to $W$, $\alpha$ many times.
Formally, we have:
\begin{definition}
    Given a $\ck$-model $M$, the approximants for $\mu X.\varphi$ and $\mu X.\varphi$ on $M$ are defined by:
    \begin{itemize}
    \item $\|\mu X^0.\varphi\|^M = \emptyset$, $\|\nu X^0.\varphi\|^M = W$;
    \item $\|\mu X^{\alpha+1}.\varphi\|^M = \|\varphi(\|\mu X^\alpha.\varphi\|^M)\|$, $\|\nu X^{\alpha+1}.\varphi\|^M = \|\varphi(\|\nu X^\alpha.\varphi\|^M)\|$, where $\alpha$ is any ordinal; and
    \item $\|\mu X^\lambda.\varphi\|^M = \bigcup_{\alpha<\lambda} \|\mu X^\alpha.\varphi\|^M$, $\|\nu X^\lambda.\varphi\|^M = \bigcap_{\alpha<\lambda} \|\nu X^\alpha.\varphi\|^M$, where $\lambda$ is a limit ordinal.
\end{itemize}
\end{definition}
Approximants will play a central role in the proof of the correctness of the game semantics.
Note that, for all $\ck$-model $M$ and $\mu$-formula $\varphi(X)$ with $X$ positive there are ordinal numbers $\alpha$ and $\beta$ such that $\|\mu X.\varphi\|^M = \|\mu X^\alpha.\varphi\|^M$ and $\|\nu X.\varphi\|^M = \|\nu X^\beta.\varphi\|^M$.

The $\ck$-models provide semantics for the basic constructive modal logic $\ck$.
The logic $\ck$ is the least set of formulas containing:
\begin{itemize}
    \item all intuitionistic tautologies;
    \item $K_\Box := \Box(\varphi\to\psi) \to (\Box\varphi \to \Box \psi)$;
    \item $K_\Diamond := \Box(\varphi\to\psi) \to (\Diamond\varphi \to \Diamond \psi)$;
\end{itemize}
and closed under necessitation and \emph{modus ponens}:
\[
    (\mathbf{Nec}) \; \frac{\varphi}{\Box\varphi} \;\;\;\;\;\text{ and }\;\;\;\;\;
    (\mathbf{MP}) \; \frac{\varphi \;\;\; \varphi\to\psi}{\psi}.
\]
The logic $\ck$ is a complete and sound axiomatization of the $\ck$-models:
\begin{theorem}[Mendler, de Paiva \cite{mendler2005constructive}]
    For all modal formula $\varphi$, $\ck$ proves $\varphi$ iff $\varphi$ is true at all $\ck$-models.
\end{theorem}

\paragraph{The modal logic $\isf$}
We now introduce $\isf$, the basic modal logic for sections \ref{sec::collapse} and \ref{sec::mu-is5}.
$\isf$ was first studied by first studied by Prior \cite{prior1957time}.
We obtain $\isf$ by adding to $\ck$ the axioms:
\begin{itemize}
    \item $FS := (\Diamond \varphi \to \Box\psi) \to \Box(\varphi\to\psi)$;
    \item $DP := \Diamond (\varphi\lor\psi) \to \Diamond\varphi\lor\Diamond\psi$;
    \item $N := \neg \Diamond \bot$;
    \item $T := \Box\varphi\to \varphi \land \varphi\to \Diamond\varphi$;
    \item $4 := \Box\varphi\to \Box\Box\varphi \land \Diamond\Diamond\varphi\to \Diamond\varphi$; and
    \item $5 := \Diamond\varphi\to \Box\Diamond\varphi \land \Diamond\Box\varphi\to \Box\varphi$;
\end{itemize}
and taking the closuse under $\mathbf{Nec}$ and $\mathbf{MP}$.
Note that $K_\Diamond$, $FS$, $DP$, and $N$ are all equivalent to $K_\Box$ in the classical setting.
This is not the case in constructive modal logic; see \cite{das2023diamonds} and \cite{degroot2024semantical} for more information.

\begin{definition}
    A $\isf$ \emph{model} is a $\ck$-model $M=\tuple{W, W^\bot, \preceq, \equiv, V}$ where: $\equiv$ is an equivalence relation over $W$; $W^\bot = \emptyset$; and the relation $\equiv$ is forward and backward confluent.
    The relation $\equiv$ is forward confluent iff $w\equiv  v$ and $w\preceq w'$ implies there is $v'$ such that $v\preceq w' \equiv v'$.
    The relation $\equiv$ is backward confluent iff $w\equiv  v \preceq v'$ implies there is $w'$ such that $w\preceq w' \equiv v'$.
    Forward and backward confluence are illustrated in Figure \ref{figure::backward-confluency}.
    We denote the modal relation by $\equiv$ instead of $R$ to emphasize that it is an equivalence relation.
\end{definition}
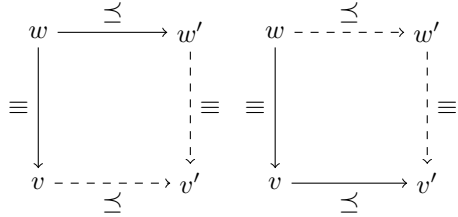
\begin{figure}[ht]
\centering
\tikzstyle{world}=[circle,draw,minimum size=5mm,inner sep=0pt]
\begin{tikzpicture}
    \node (w) at (0,0) {$w$};
    \node (w2) at (2,0) {$w'$};
    \node (v) at (0,-2) {$v$};
    \node (v2) at (2,-2) {$v'$};

    \draw[->] (w) -- (w2) node[midway,above] {$\preceq$};
    \draw[->] (w) -- (v) node[midway,left] {$\equiv$};
    \draw[dashed,->] (v) -- (v2) node[midway,below] {$\preceq$};
    \draw[dashed,->] (w2) -- (v2) node[midway,right] {$\equiv$};
\end{tikzpicture}
\begin{tikzpicture}
    \node (w) at (0,0) {$w$};
    \node (w2) at (2,0) {$w'$};
    \node (v) at (0,-2) {$v$};
    \node (v2) at (2,-2) {$v'$};

    \draw[dashed,->] (w) -- (w2) node[midway,above] {$\preceq$};
    \draw[->] (w) -- (v) node[midway,left] {$\equiv$};
    \draw[->] (v) -- (v2) node[midway,below] {$\preceq$};
    \draw[dashed,->] (w2) -- (v2) node[midway,right] {$\equiv$};
\end{tikzpicture}
\caption{Schematics for forward and backward confluence.}
\label{figure::backward-confluency}
\end{figure}

The modal logic $\isf$ is a complete and sound axiomatization of the $\isf$-models:
\begin{theorem}[Ono \cite{ono1977intuitionistic}, Fischer Servi \cite{fischerservi1978finite}]
    For all modal formula $\varphi$, $\isf$ proves $\varphi$ iff $\varphi$ is true at all $\isf$-models.
\end{theorem}

The following lemma will be useful when proving the $\mu$-calculus' collapse over $\isf$-models:
\begin{lemma}
    \label{lem::transitiveness-is5}
    Let $M=\tuple{W, W^\bot, \preceq, \equiv, V}$ be an $\isf$-model.
    Then the composition $\preceq;\equiv$ is a transitive relation.
\end{lemma}
\begin{proof}
    Suppose $w\preceq w'\equiv v \preceq v'\equiv u$.
    By backward confluence, there is $u'$ such that $w'\preceq u'\equiv v'$.
    By the transitivity of $\preceq$ and $\equiv$, $w\preceq u'\equiv u$.
\end{proof}

In general, forward confluence and backward confluence is are not equivalent over $\ck$-models; but if the modal relation is an equivalence relation, then they coincide.
\begin{proposition}
    \label{prop::confluences-are-equivalent}
    Let $M=\tuple{W, W^\bot, \preceq, \equiv, V}$ be a $\ck$-model where $\equiv$ is an equivalence relation over $W$.
    Then the modal relation $\equiv$ is forward confluent iff $\equiv$ is backward confluent.
\end{proposition}
\begin{proof}
    Suppose $\equiv$ is forward confluent.
    As $\equiv$ is an equivalence relation, $v\equiv w$ and $v\preceq v'$.
    By forward confluence, there is $w'$ such that $w\preceq w'$ and $v'\equiv w'$.
    Therefore $w \preceq w' \equiv v'$.
    The proof that backwards confluence implies forward confluence is similar.
\end{proof}

\paragraph{The fixed-point logic $\muisf$}
We obtain $\muisf$ by adding to $\isf$ the fixed-point axioms:
\begin{itemize}
    \item $\nu FP := \nu X.\varphi \to \varphi(\nu X.\varphi)$; and
    \item $\mu FP := \varphi(\mu X.\varphi) \to \mu X.\varphi$;
\end{itemize}
and taking the closure under $\mathbf{Nec}$, $\mathbf{MP}$ and the induction rules:
\[
    (\mathbf{\nu Ind}) \; \frac{\psi\to \varphi(\psi)}{\psi \to \nu X.\varphi} \;\;\;\;\;\text{ and }\;\;\;\;\;
    (\mathbf{\mu Ind}) \; \frac{\varphi(\psi) \to \psi}{\mu X.\varphi \to \psi},
\]
where $X$ is positive in $\varphi$.
Note that the two fixed-point axioms and the two induction rules are necessary as $\nu$ and $\mu$ cannot be defined in terms of each other over constructive semantics.
While, over classical semantics, $\nu X.\varphi$ is equivalent to $\neg \mu X.\neg\varphi(\neg X)$; we do not have the same in constructive semantics.
If $\varphi := P$, then $\nu X.\varphi$ is equivalent to $P$ and $\neg \mu X.\neg\varphi(\neg X)$ is equivalent to $\neg\neg P$; but $P$ and $\neg\neg P$ are not equivalent formulas.

\section{Game semantics for the constructive \texorpdfstring{$\mu$}{mu}-calculus}
\label{sec::game-semantics}
In this section, we define game semantics for the constructive $\mu$-calculus and prove its equivalence to the bi-relational semantics.
This game semantics is a modification of the game semantics of the classical $\mu$-calculus.
In the classical version, the players Verifier and Refuter discuss whether a formula $\varphi$ hods in a world $w$ of a Kripke model $M$.
While in the classical $\mu$-calculus we can suppose formulas use no implications and that negations are applied only to propositional symbols, we cannot do the same in the constructive $\mu$-calculus.
This complicates the games used for the constructive $\mu$-calculus: the players now have the \emph{roles} of Refuter and Verifier, and swap roles when discussing certain formulas.

\subsection{Definition}
Fix a $\ck$-model $M=\tuple{W,W^\bot,\preceq,R,V}$, a world $w\in W$, and a well-named $\mu$-formula $\varphi$.
In this subsection, we define the \emph{evaluation game} $\mathcal{G}(M,w\models\varphi)$.

The game $\mathcal{G}(M,w\models\varphi)$ has two players: $\mathsf{I}$ and $\mathsf{II}$.
The two players will have the roles of Verifier and Refuter (abbreviated to $\verifier$ and $\refuter$, respectively).
Each player has only one role at any given time, and the players always have different roles.
We usually write ``$\mathsf{I}$ is $\verifier$'' for ``the player in the role of $\verifier$'', and similar expressions for the other combination of players and games.
We denote an arbitrary role by $\role$ and the dual role by $\dualrole$; that is, if $\role$ is $\verifier$, then $\dualrole$ is $\refuter$ and \emph{vice versa}.

The game has two types of positions.
The main positions of the game are of the form $\tuple{v, \psi, \role}$ where $v\in W$, $\psi\in\sub{\varphi}$, and $\role$ is a role.
We also have auxiliary positions of the form $\tuple{\tuple{v}, \psi, \role}$ if $\Diamond\psi\in\sub{\varphi}$, $\tuple{[v], \psi, \role}$ if $\Box\psi\in\sub{\varphi}$, and $\tuple{v, \theta?\theta', \role}$ if $\theta\to\theta'\in\sub{\varphi}$, where $v\in W$.
In any position, $\role$ is the role currently held by $\mathsf{I}$; the role of $\mathsf{II}$ is $\dualrole$.
Intuitively, at a position $\tuple{v,\psi, \role}$, $\verifier$ tries to show that $v$ satisfies $\psi$ and $\refuter$ tries to prove that $v$ does not satisfy $\psi$.

These auxiliary positions are used to decompose the players moves at positions of the forms $\tuple{v,\Diamond\psi, \role}$, $\tuple{v,\Box\psi, \role}$, and $\tuple{v,\psi\to\theta, \role}$. For example, at the position $\tuple{v,\Diamond\psi, \role}$, $\refuter$ first makes a choice and then $\verifier$ does so.
The auxiliary positions make explicit that both players are involved in the choice of the next main position $\tuple{u,\psi, \role}$.
For the same reasoning, we also use auxiliary positions is also necessary for $\tuple{v,\theta\to\theta', \role}$.
We use auxiliary positions for $\tuple{v,\Box\psi, \role}$ for uniformity's sake.

The game begin at the position $\tuple{w,\varphi,\verifier}$, with $\mathsf{I}$ in the role of $\verifier$ and $\mathsf{II}$ in the role of $\refuter$.

Each position is owned by exactly one of the players.
At each turn of the game, the owner of the current position has chooses one of the available positions to move to.
The game then continues with the new position.
If no such position is available, the game ends.
We describe the ownership and possible plays for each type of position below; this information is summarized in Table \ref{table::intuitionistic-evaluation_game}.

\begin{table}[htb]\renewcommand\arraystretch{1.2}
    \caption{Rules of evaluation games for the constructive modal $\mu$-calculus.}
    \begin{center}
    \begin{tabular}{c|c}
    \multicolumn{2}{c}{Verifier}\\
    \hline
    Position &  Admissible moves\\

    $\tuple{v, \psi \lor \theta, \role}$ &  $\{\tuple{v, \psi, \role} , \tuple{v,\theta, \role} \}$\\

    $\tuple{v,\psi?\theta, \role}$ & $\{\tuple{v,\psi, \dualrole}, \tuple{v,\theta, \role}\}$ \\

    $\tuple{\tuple{v}, \psi, \role}$ & $\{ \tuple{u, \psi, \role} \mid vR u \}$ \\

    $\tuple{v,P, \role}$ and $v \not\in V(P)$  &  $\emptyset$ \\

    $\tuple{v,\mu X.\psi_X, \role}$ &  $\{\tuple{v,\psi_X, \role} \}$ \\

    $\tuple{v,X, \role}$ & $\{ \tuple{v,\mu X.\psi_X, \role} \}$ \\
    \end{tabular} \\
    \begin{tabular}{c|c}
    \multicolumn{2}{c}{Refuter}\\
    \hline
    Position &  Admissible moves\\

    $\tuple{v,\psi \land \theta, \role}$ &  $\{ \tuple{v,\psi, \role} , \tuple{v,\theta, \role} \}$ \\

    $\tuple{v,\neg\psi, \role}$ & $\{\tuple{u,\psi, \dualrole} \mid v\preceq v\}$ \\
    $\tuple{v,\psi\to\theta, \role}$ & $\{\tuple{u,\psi?\theta, \role} \mid v\preceq v\}$ \\

    $\tuple{v, \Diamond \psi, \role}$ & $\{ \tuple{\tuple{u}, \psi, \role} \mid v\preceq u \}$ \\

    $\tuple{v, \Box \psi, \role}$ & $\{ \tuple{[u], \psi, \role} \mid v\preceq u\}$ \\

    $\tuple{[v],\psi, \role}$ & $\{ \tuple{u, \psi, \role} \mid vR u \}$ \\

    $\tuple{v,P, \role}$ and $v \in V(P)$  &  $\emptyset$ \\

    $\tuple{v,\nu X.\psi_X, \role}$ &  $\{\tuple{v,\psi_X}, \role\}$ \\

    $\tuple{v,X, \role}$ & $\{ \tuple{v,\nu X.\psi_X, \role} \}$ \\

    $\tuple{v,\psi, \role}$, $v\in W^\bot$ and $\psi\in\sub{\varphi}$ & $\emptyset$ \\
    \end{tabular}
    \end{center}
    \label{table::intuitionistic-evaluation_game}
\end{table}

If $v\in W^\bot$ and $\psi\in \sub{\varphi}$, then the position $\tuple{v,\psi, \role}$ is owned by $\refuter$ and there is no available move.
Below, we suppose the world $v$ in the position $\tuple{v,\psi,\role}$ being described is not in $W^\bot$.

At the position $\tuple{v,P, \role}$ there is no available move and the game ends.
This position is owned by $\refuter$ if $v\models P$ and by $\verifier$ if $v\not\models P$.

The position $\tuple{v, \psi\lor\theta, \role}$ is owned by $\verifier$, who chooses one of $\tuple{v,\psi, \role}$ and $\tuple{v,\theta, \role}$.
Similarly, at $\tuple{v, \psi\land\theta, \role}$ is owned by $\refuter$, who chooses one of $\tuple{v,\psi, \role}$ and $\tuple{v,\theta, \role}$.

The position of the form $\tuple{v,\Box\psi, \role}$ is owned by $\refuter$, who chooses $v'$ such that $v\preceq v'$, and then move to the position $\tuple{[v'],\psi, \role}$; the position $\tuple{[v],\psi, \role}$ is again owned by $\refuter$, who chooses $v''$ such that $v'R v''$ and moves to $\tuple{v'',\psi, \role}$.
Similarly, the position $\tuple{v,\Diamond\psi, \role}$ is owned by $\refuter$ chooses $v'$ such that $v\preceq v'$, and then move the position $\tuple{\tuple{v'},\psi, \role}$; the position $\tuple{\tuple{v'},\psi, \role}$ is owned by $\verifier$ chooses $v''$ such that $v'R v''$ and moves to $\tuple{v'',\psi, \role}$.

At a position of the form $\tuple{v,\neg\psi, \role}$, $\refuter$ chooses $v'\succeq v$ and challenges $\verifier$ to show that $M,v'\not\models\varphi$; that is, $\refuter$ moves to $\tuple{v',\psi, \dualrole}$.
Positions of the form $\tuple{v,\psi\to\theta, \role}$ is similar.
In this case, $\refuter$ chooses $v'\succeq v$ and moves to $\tuple{v',\psi?\theta, \role}$, and then $\verifier$ chooses one of $\tuple{v',\psi, \dualrole}$ and $\tuple{v',\theta, \role}$.
That is, $\refuter$ chooses $v'\succeq v$ and $\verifier$ chooses whether to show that $M,v'\not\models \psi$ or $M,v'\models \theta$; in case $\verifier$ chooses $\tuple{v',\psi, \dualrole}$, the players exchange roles.

Let $\eta X.\psi_X\in \sub{\varphi}$; at the positions $\tuple{v,\eta X.\psi_X, \role}$ and $\tuple{v,X, \role}$ are owned by $\verifier$ if $\eta$ is $\nu$ and by $\refuter$ if $\eta$ is $\mu$; the only available position to move to is $\tuple{v,\psi_X, \role}$.
When moving from $\tuple{v,X, \role}$ to $\tuple{v,\nu X.\psi_X, \role}$, we say that the fixed-point formula $\eta X.\psi$ was \emph{regenerated}.

A \emph{run} of the game is a sequence of positions which respects the rules above.
That is a run $\rho$ is a (finite or infinite) sequence $\tuple{v_0,\psi_0,\role_0}, \tuple{v_1,\psi_1,\role_1}, \tuple{v_2,\psi_2,\role_2}, \dots$ of position such that:
\begin{itemize}
    \item $\tuple{v_0,\psi_0,\role_0}$ is $\tuple{w,\varphi,\verifier}$;
    \item it is possible to play $\tuple{v_{n+1},\psi_{n+1},\role_{n+1}}$ from $\tuple{v_n,\psi_n,\role_n}$; and
    \item if $\rho$ is finite, then the last position in $\rho$ is of the form $\tuple{v,P,\role}$ with $P\in\mathrm{Prop}$ or $\tuple{v,\psi,\role}$ with $v\in W^\bot$.
\end{itemize}
Before defining the winning conditions, we note that the positivity requirement on the fixed-point formulas guarantees that, if $\tuple{v,\eta X.\psi_X, \role}$ and $\tuple{v',\eta X.\psi_X, \role'}$ occur in any run of the game, then $\mathsf{I}$ has the same role in both positions:
\begin{proposition}
    Let $\rho$ and $\rho'$ be a runs of the game $\mathcal{G}(M,w\models\varphi)$.
    Suppose $\tuple{v,\eta X. \psi_X, \role}$ occurs in $\rho$ and if $\tuple{v',\eta X. \psi_X, \role'}$ occurs in $\rho'$.
    Then $\role = \role'$; that is, $\mathsf{I}$ has the same role at both positions.
\end{proposition}
\begin{proof}
    If $\tuple{v,\eta X. \psi_X, \role}$ and $\tuple{v',\eta X. \psi_X, \role'}$ occur in the same run $\rho$, then $\role$ and $\role'$ coincide by the positivity of $X$ in $\psi_X$: it implies that the players must swap roles an even number of times between these two positions.

    Now, let $\tuple{v,\eta X. \psi_X, \role}$ and $\tuple{v',\eta X. \psi_X, \role'}$ be the first occurrence of positions with the formula $\eta X.\psi_X$ in $\rho$ and $\rho'$, respectively.
    The well-namedness of $\varphi$ implies that there is only one occurrence of $\eta X$ in $\varphi$.
    This fact along with the positivity of $X$ implies that the number of times the players switch roles to get to $\tuple{v,\eta X. \psi_X, \role}$ and $\tuple{v',\eta X. \psi_X, \role'}$ must have the same parity.
\end{proof}
We say that the fixed-point formula \emph{$\eta X.\psi_X$ is owned by $\mathsf{I}$} if a position of the form $\tuple{v,\eta X.\psi_X,\role}$ is reachable from the initial position either $\role = \verifier$ and $\eta = \nu$, or if $\role = \refuter$ and $\eta = \mu$.
The fixed-point formula $\eta X.\psi_X$ is \emph{owned by $\mathsf{II}$} if it is not owned by $\mathsf{I}$.
This is well-defined by the above proposition.

We are now ready to define the winning conditions for the game.
Let $\rho$ be a run of the game.
If $\rho$ is finite, then the last position in $\rho$ is of the form $\tuple{v,P,\role}$ with $P\in\mathrm{Prop}$ or $\tuple{v,\psi,\role}$ with $v\in W^\bot$.
The owner of the last position has no available position and loses the game.
If $\rho$ is infinite, let $\eta X.\psi_{X}$ outermost infinitely often regenerated fixed-point formula in the play $\rho$; that is, $\eta X$ is regenerated infinitely often in $\rho$ and, if $\eta' Y$ is regenerated infinitely often in $\rho$, then $\eta' Y.\psi_Y\in\mathrm{Sub}(\eta X.\psi_X)$.
Then $\mathsf{I}$ wins $\rho$ iff $\mathsf{I}$ owns the fixed-point $\eta X$.

A \emph{(positional) strategy} for $\mathsf{I}$ is a function $\sigma$ which, given a position $\tuple{v,\psi, \role}$ owned by $\mathsf{I}$, outputs a position $\tuple{v',\psi', \role}$ where $\mathsf{I}$ can move to, if any such position is available.
$\mathsf{I}$ follows $\sigma$ in the run $\rho = \tuple{v_0,\psi_0,\role_0}, \tuple{v_1,\psi_1,\role_1}, \tuple{v_2,\psi_2,\role_2}, \dots$ if whenever $\tuple{v_n}, \psi_n, \role_n$ is owned by $\mathsf{I}$, then $\tuple{v_{n+1}}, \psi_{n+1}, \role_{n+1} = \sigma(\tuple{v_n}, \psi_n, \role_n)$.
The strategy $\sigma$ is \emph{winning} iff $\mathsf{I}$ wins all possible runs where they follow $\sigma$.
Strategies and winning strategies for $\mathsf{II}$ are defined similarly.

Note that at most one of the player can have a winning strategy for a given evaluation game.
It is not immediate that that one of the players has a winning strategy for a given evaluation game.
The existence of winning strategies is implied by our proof of the equivalence of the bi-relational Kripke semantics and game semantics for the constructive $\mu$-calculus:
\begin{theorem}
    Fix a $\ck$-model $M=\tuple{W,W^\bot,\preceq,R,V}$, a world $w\in W$, and a well-named $\mu$-formula $\varphi$.
    Let $\mathcal{G}(M,w\models\varphi)$ be an evaluation game.
    Then only one $\mathsf{I}$ and $\mathsf{II}$ has a positional winning strategy $\mathcal{G}(M,w\models\varphi)$.
\end{theorem}
\begin{proof}
    It is immediate that not both $\mathsf{I}$ and $\mathsf{II}$ have winning strategies for $\mathcal{G}(M,w\models\varphi)$.
    For a contradiction, suppose $\sigma$ is a winning strategy for $\mathsf{I}$ and $\tau$ is a winning strategy for $\mathsf{II}$.
    Then the play resulting from the players using $\sigma$ and $\tau$ is winning for both $\mathsf{I}$ and $\mathsf{II}$, which is not possible by the definition of the game.

    Now, by the definition of the bi-relational semantics, either $M,w\models\varphi$ or $M,w\not\models\varphi$.
    Theorem \ref{thm::corectness-game-semantics} provides us strategies in both cases.
    In case $M,w\models\varphi$ holds, $\mathsf{I}$ has a winning strategy; in case $M,w\not\models\varphi$ holds, $\mathsf{II}$ has a winning strategy.
\end{proof}
Note that the existence of positional strategies for the evaluation games is not a trivial fact.
A key fact in the existence of such strategies is that, in order to determine the winner of an infinite run $\rho$, we need to look only at its tails; the exclusion of any initial segment of the run does not alter the result.
Another way of proving the existence of winning strategies for evaluation games is to represent them as parity games; the existence of positional winning strategies for parity games was proved by Emerson and Jutla \cite{emerson1991tree}.
The relation between the $\mu$-calculus and parity games is outside the scope of this paper, see \cite{gradel2003automata} for more information.

\subsection{Correctness of game semantics}
We now show the equivalence between the $\mu$-calculus' bi-relational semantics and game semantics.
That is, we will show that $M,w\models\varphi$ iff the player $\mathsf{I}$ has a winning strategy for $\mathcal{G}(M,w\models\varphi)$, and that $M,w\not\models\varphi$ iff the player $\mathsf{II}$ has a winning strategy for $\mathcal{G}(M,w\models\varphi)$.
Before proving it, we briefly sketch the key idea and remark on some technical points.

Fix an evaluation game $\mathcal{G}(M,w\models\varphi)$.
Call a position $\tuple{v,\psi, \role}$ \emph{true} iff $M,w\models \psi$, and \emph{false} iff $M,w\not\models\psi$.
We will show that, at a true positions, $\verifier$ can always move in a way favorable to themselves.
That is, in a way such that the resulting position is a true position, if the players have not switched roles; or the resulting position is a false position, if the players have switched roles.
On the other hand, $\refuter$ cannot move in any way favorable to themselves.
A similar situation occurs at false positions.
To make the statements above precise, we have to overcome two problems.

First, when considering whether $M, v\models\psi$, the formula $\psi$ might have free variables, and so its valuation might not be well-defined.
We solve this by augmenting $M$ with the intended valuations for the variables occurring in $\varphi$.

Second, we need to consider infinite plays.
Specifically, we need to guarantee that, when starting from a true position, the resulting play is winning for $\mathsf{I}$; and when starting from a false position, the resulting play is winning for $\mathsf{II}$.
To solve this, we assign two types of signatures to each position $\tuple{v,\psi, \role}$.
We will show that the players can play such that, if we start from a true position, the $\mathsf{I}$-signatures are non-increasing and eventually constant; and if we start from a false position, the $\mathsf{II}$-signatures are non-increasing and eventually constant.
This will guarantee the resulting plays are winning for the corresponding players.

\begin{theorem}
    \label{thm::corectness-game-semantics}
    Let $M=\tuple{W, W^\bot,\preceq,R,V}$ be a $\ck$-model, $w\in W$ and $\varphi$ be a well-named $\mu$-formula.
    Then
    \begin{align*}
        \mathsf{I} \text{ has a winning strategy for } \mathcal{G}(M,w\models\varphi) &\text{ if and only if } M,w\models\varphi\text{, and } \\
        \mathsf{II} \text{ has a winning strategy for } \mathcal{G}(M,w\models\varphi) &\text{ if and only if } M,w\not\models\varphi.
    \end{align*}
\end{theorem}

\begin{proof}
    We prove that, if $M,w\models\varphi$, then $\mathsf{I}$ has a winning strategy for $\mathcal{G}(M,w\models\varphi)$ and, if $M,w\not\models\varphi$, then $\mathsf{II}$ has a winning strategy for $\mathcal{G}(M,w\models\varphi)$.
    This is sufficient to prove the theorem since the two players cannot both have a winning strategy for $\mathcal{G}(M,w\models\varphi)$ and since one of $M,w\models\varphi$ or $M,w\not\models\varphi$ always holds.

    Suppose $w\models \varphi$.
    We will assign to each main position $\tuple{v,\psi, \role}$ of the game an ordinal signature $\sig{v,\psi, \role}$.
    We show $\mathsf{I}$ is always able to control the truth of the positions in the evaluation game $\mathcal{G}(M,w\models\varphi)$ and move in a way that the signature is eventually constant.

    To define $\mathsf{I}$-signatures, we enumerate the fixed-point subformulas of $\varphi$ in non-increasing size:
    \[
        \eta_1 Z_1.\psi_1, \eta_2 Z_2.\psi_2, \dots, \eta_n Z_n.\psi_n.
    \]
    That is, we require that, if $i<j$, then $\eta_{i+1} Z_{i+1}.\psi_{i+1}\not\in \mathrm{Sub}(\eta_j Z_j.\psi_j)$; and, if $\eta_{i+1} Z_{i+1}.\psi_{i+1}\in \mathrm{Sub}(\eta_j Z_j.\psi_j)$, then $j\leq i$.
    We also enumerate the fixed-point subformulas of $\varphi$ which are owned by $\mathsf{I}$ in non-increasing size:
    \[
        \eta_1' Y_1.\chi_1, \eta_2' Y_2.\chi_2, \dots, \eta_m' Y_m.\chi_m.
    \]

    An \emph{$\mathsf{I}$-signature} $r = \tuple{r(1), \dots, r(m)}$ is a sequence of $m$ ordinals.
    Denote by $r(k)$ the $k$th component of $r$.
    Write $r =_k r'$ iff the first $k$ components of $r$ are identical.
    Order the signatures by the lexicographical order: $r<r'$ iff there is $k\in\{1,\dots, m\}$ such that $r =_k r'$ and $r(k+1)<r'(k+1)$.
    The lexicographical order is a well-ordering of the signatures.

    The \emph{augmented Kripke model} $M[X\mapsto A]$ is obtained by setting $V(X) := A$, where $M = \tuple{W, W^\bot,\preceq, R,V}$ is a Kripke model, $A\subseteq W$ and $X$ is a variable symbol.
    We want to evaluate subformulas of $\varphi$ where some $Z_1,\dots, Z_n$ occur free, so we augment $M$ with the correct valuations of these variables:
    \begin{align*}
        M_0     &:= V; \\
        M_{i+1} &:= M_i[Z_{i+1}\mapsto \|\eta_{i+1} Z_{i+1}.\psi_{i+1}\|^{M_i}].
    \end{align*}
    By the choice of our enumeration, $\eta_{i+1} Z_{i+1}.\psi_{i+1}$ does not contain free occurrences of $Z_{i+1},\dots, Z_n$, and so $M_i$ is well-defined.

    The $\alpha$th approximant $\mu X^\alpha.\varphi$ of $\mu X.\varphi$ is obtained by applying $\Gamma_\varphi(X)$ to $\emptyset$, $\alpha$ many times; and that the approximant $\nu X^\alpha.\varphi$ is obtained by applying $\Gamma_\varphi(X)$ to $W$, $\alpha$ many times.
    We define models $M_n^r$ where the variables $Y_j$ owned by $\mathsf{I}$ are assigned their $r(j)$th approximant $\|\eta_j^{r(j)} Y_j.\chi_j\|$, and variables owned by $\mathsf{II}$ receive their correct value.
    Formally, given a signature $r$, we define augmented models $M^r_0, \dots, M^r_n$ by
    \begin{align*}
        M_0^r     &:= V; \\
        M_{i+1}^r &:= \left\{
        \begin{array}{ll}
            M_i[Z_{i+1}\mapsto \|\eta_j' Y_j^{r_j}.\chi_j\|^{M_i^r}], & \text{if $ Z_{i+1} = Y_j$}; \\
            M_i[Z_{i+1}\mapsto \|\eta_{i+1} Z_{i+1}.\psi_{i+1}\|^{M_i}], & \text{if there is no $j$ such that $Z_{i+1} = Y_j$}.
        \end{array} \right.
    \end{align*}

    If $M_n,v\models\psi$, we call $\tuple{v,\psi, \role}$ a \emph{true position}; if $M_n,v \not\models\psi$, we call $\tuple{v,\psi, \role}$ a \emph{false position}.
    Now, if $\tuple{v,\psi, \role}$ a true position, then there is a least signature $r$ such that $M_n^r,v\models\psi$.
    Similarly, if $\tuple{v,\psi, \role}$ a false position, then there is a least signature $r$ such that $M_n^r,v\not\models\psi$.
    Denote these signatures by $\sig{v,\psi, \role}$.

    We will define a strategy $\sigma$ for $\mathsf{I}$ which guarantees that when the players are at $\tuple{v,\psi, \role}$, $v\models\psi$ if $\mathsf{I}$ is in the role of $\verifier$, and $v\not\models\psi$ if $\mathsf{I}$ is in the role of $\refuter$.
    Furthermore, $\mathsf{II}$ cannot move in ways the signature is increasing and most of $\mathsf{I}$'s moves never increase the signature.
    The only time the signature may increase is when regenerating some fixed-poin formula $\eta'_jY_j.\chi_j$, but in this case the first $j-1$ positions of the signature are not modified.

    We will also have that any positions $\tuple{v,\psi,\role}$ reachable when $\mathsf{I}$ follows the strategy $\sigma$ are true positions when $\role = \verifier$ and false positions when $\role=\refuter$.
    Remember that the game starts on the position $\tuple{w,\varphi,\verifier}$ and we assumed that $M,w\models\varphi$ holds, so this is true for the initial position of the game.

    We define $\mathsf{I}$'s strategy as follows:
    \begin{itemize}
        \item Suppose the game is at the position $\tuple{v,\psi_1\lor\psi_2, \role}$.
        If $\role = \verifier$ and $\tuple{v,\psi_1\lor\psi_2, \verifier}$ is a true position; then $\mathsf{I}$ moves to $\tuple{v,\psi_i, \verifier}$ such that $M_n^\sig{v,\psi_1\lor\psi_2, \verifier}, v\models \psi_i$, with $i\in\{1,2\}$.
        By the definition of the signatures, $\sig{v,\psi_1\lor\psi_2, \verifier} = \sig{v,\psi_i, \verifier}$.
        If $\role = \refuter$ and $\tuple{v,\psi_1\lor\psi_2,\refuter}$ is a false position; then $M_n^\sig{v,\psi_1\lor\psi_2,\refuter}, v\not\models \psi_i$ and $\sig{v,\psi_1\lor\psi_2,\refuter} \geq \sig{v,\psi_i,\refuter}$ for all $i\in\{1,2\}$. So whichever way $\mathsf{II}$ moves, the next position is false and the signature is non-increasing.

        \item Suppose the game is at the position $\tuple{v,\psi_1\land\psi_2,\role}$.
        If $\role = \verifier$ and $\tuple{v,\psi_1\land\psi_2,\verifier}$ is a true position; then $M_n^\sig{v,\psi_1\land\psi_2,\verifier}, v\models \psi_i$ and $\sig{v,\psi_1\land\psi_2,\verifier} \geq \sig{v,\psi_i}$ for all $i\in\{1,2\}$.
        So whichever way $\mathsf{II}$ moves, the next position is true and the signature is non-increasing.
        If $\role = \refuter$ and $\tuple{v,\psi_1\land\psi_2,\refuter}$ is a false position; then $\mathsf{I}$ moves to $\tuple{v,\psi_i,\refuter}$ such that $M_n^\sig{v,\psi_1\land\psi_2,\refuter}, v\not\models \psi_i$ and $\sig{v,\psi_1\land\psi_2,\refuter} = \sig{v,\psi_i,\refuter}$, with $i\in\{1,2\}$.

        \item Suppose the game is at the position $\tuple{v,\Diamond\psi,\role}$.
        If $\role = \verifier$ and $\tuple{v,\Diamond\psi,\verifier}$ is a true position; for all move $\tuple{\tuple{v'},\psi,\verifier}$ of $\mathsf{II}$, $\mathsf{I}$ can move to some $\tuple{v'',\psi,\verifier}$ such that $M_n^\sig{v,\Diamond\psi,\verifier}, v''\models \psi$.
        By the definition of the signatures, $\sig{v,\Diamond\psi,\verifier} \geq \sig{v'',\psi,\verifier}$.
        If $\role = \refuter$ and $\tuple{v,\Diamond\psi,\refuter}$ is a false position; $\mathsf{I}$ moves to a position $\tuple{\tuple{v'},\psi,\refuter}$ such that all answers $\tuple{v'',\psi,\refuter}$ by $\mathsf{II}$ are false positions.
        Furthermore, $M_n^\sig{v,\Diamond\psi,\refuter}, v''\not\models \psi$ and $\sig{v,\Diamond\psi,\refuter} \geq \sig{v'',\psi,\refuter}$ for all such $v''$.

        \item Suppose the game is at the position $\tuple{v,\Box\psi,\role}$.
        If $\role = \verifier$ and $\tuple{v,\Box\psi,\verifier}$ is a true position; for all moves $\tuple{[v'],\psi,\verifier}$ and $\tuple{v'',\psi,\verifier}$ of $\mathsf{II}$, we have $M_n^\sig{v,\Box\psi,\verifier}, v''\models \psi$.
        By the definition of the signatures, $\sig{v,\Box\psi,\verifier} \geq \sig{v'',\psi,\verifier}$.
        If $\role = \refuter$ and $\tuple{v,\Box\psi,\refuter}$ is a false position; $\mathsf{I}$ moves to a position $\tuple{[v'],\psi,\refuter}$ and then to a position $\tuple{v',\psi,\refuter}$ which is a false position.
        Furthermore, $M_n^\sig{v,\Box\psi,\refuter}, v''\not\models \psi$ and $\sig{v,\Box\psi,\refuter} \geq \sig{v'',\psi}$.

        \item Suppose the game is at the position $\tuple{v,\neg\psi,\role}$.
        If $\role = \verifier$ and $\tuple{v,\neg\psi,\verifier}$ is a true position; after all move $\tuple{v',\psi,\refuter}$ of $\mathsf{II}$, the players switch roles and we have $M_n^\sig{v,\neg\psi,\verifier}, v'\not\models \psi$.
        By the definition of the signatures, $\sig{v,\neg\psi,\verifier} \geq \sig{v',\psi,\refuter}$.
        If $\role = \refuter$ and $\tuple{v,\Box\psi,\refuter}$ is a false position; $\mathsf{I}$ moves to a position $\tuple{v',\psi,\verifier}$ which is a true position and switches roles with $\mathsf{II}$.
        Furthermore, $M_n^\sig{v,\neg\psi,\refuter}, v'\models \psi$ and $\sig{v,\neg\psi,\refuter} \geq \sig{v',\psi,\verifier}$.

        \item Suppose the game is at the position $\tuple{v,\psi_1\to\psi_2,\role}$.
        If $\role = \verifier$ and $\tuple{v,\psi_1\to\psi_2,\verifier}$ is a true position.
        After $\mathsf{II}$ moves to $\tuple{v',\psi_1?\psi_2,\verifier}$, $\mathsf{I}$ moves to $\tuple{v',\psi_2,\verifier}$ if it is a true position.
        Otherwise, $\mathsf{I}$ moves to $\tuple{v',\psi_1,\refuter}$; in this case, $\tuple{v',\psi_1,\refuter}$ is a false position.
        Either way, $\sig{v,\psi_1\to\psi_2,\verifier} \geq \sig{v',\psi_i,\role'}$.
        If $\role = \refuter$ and $\tuple{v,\psi_1\to\psi_2,\refuter}$ is a false position; $\mathsf{I}$ moves to a position $\tuple{v',\psi_1?\psi_2,\refuter}$ such that $\tuple{v',\psi_1\verifier}$ is a true position  and $\tuple{v',\psi_2,\refuter}$ is a false position.
        Any answer of $\mathsf{II}$ satisfies our requirements.

        \item Suppose the game is at $\tuple{v,\eta_j' Y_j.\chi_j,\role}$ or at $\tuple{v,Y_j,\role}$, then the owner of the position must move to $\tuple{v,\chi_j,\role}$.
        If there there is $j$ such that $Z_i = Y_j$, then $\sig{v,\eta_j' Y_j.\chi_j} =_{j-1} \sig{v,Y_j} =_{j+1} \sig{w,\chi_j}$ and $\sig{w,Y_j}(j) > \sig{w,\chi_j}(j)$.
        If there is no $j$ such that $Z_i = Y_j$, then $\sig{w,\eta_{i+1} Z_{i+1}.\psi_{i+1}} = \sig{w,Z_i} = \sig{w,\psi_i}$.
    \end{itemize}

    On finite runs, $\mathsf{I}$ wins by the construction of the strategy $\sigma$: $\mathsf{I}$ is $\mathsf{V}$ at a true position of the form $\tuple{v,P,\role}$ reachable following $\sigma$.
    Similarly, $\mathsf{I}$ is $\mathsf{R}$ at a false positions of the form $\tuple{v,P,\role}$.
    Also, $\mathsf{I}$ is $\mathsf{V}$ at true positions $\tuple{v,\psi,\role}$ where $v\in W^\bot$.

    Now, consider an infinite run $\tuple{w_0,\varphi_0,\role_0}, \tuple{w_1, \varphi_1,\role_1}, \dots$ where $\mathsf{I}$ follows $\sigma$, let $i$ be the smallest number in $\{1,\dots, n\}$ such that $\eta_i Z_i$ is an infinitely often regenerated fixed-point operator.
    Suppose there is $j\in\{1,\dots, m\}$ such that $Z_i = Y_j$ for a contradiction.
    Let $k_1,k_2,\dots$ be the positions where $Y_j$ occur; that is, all the positions $\tuple{w_{k_l}, \psi,\role} = \tuple{w_{k_l}, Y_j,\role}$.
    Without loss of generality, we suppose that for all $i'<i$ no $Z_{i'}$ is regenerated after the $k_1$th position of the run.
    The move from $\tuple{w_{k_i}, Y_j,\role}$ to $\tuple{w_{k_{i}+1}, \chi_j,\role}$ causes a strict decrease in the signature.
    The other moves between $k_i+1$ and $k_{i+1}$ cannot cancel this decrease, since either the signature does not change or one of the first $i$ positions of the signature is reduced.
    Therefore the sequence of signatures
    \[
        \sig{w_{k_1}, Y_j,Q_{k_1}}, \sig{w_{k_2}, Y_j,Q_{k_2}}, \sig{w_{k_3}, Y_j,Q_{k_3}}, \dots
    \]
    is strictly decreasing.
    This is a contradiction, as the signatures are well-ordered.
    Therefore there is no $j$ such that $Z_i = Y_j$, and so $\mathsf{I}$ wins the run.

    We conclude that the strategy $\sigma$ is a winning strategy for $\mathsf{I}$.

    We now sketch how to prove the other half of the theorem.
    If $M,w\not\models\varphi$, then we can define a winning strategy for $\mathsf{II}$ similar to the strategy for $\mathsf{I}$ defined above.
    The main difference is that we need to consider $\mathsf{II}$-signatures, denoting approximants for $\mathsf{II}$'s variables.

    Again, enumerate the fixed-point subformulas of $\varphi$ in non-increasing size:
    \[
        \eta_1 Z_1.\psi_1, \eta_2 Z_2.\psi_2, \dots, \eta_n Z_n.\psi_n.
    \]
    We now also enumerate the fixed-point subformulas of $\varphi$ which are owned by $\mathsf{II}$ in non-increasing size:
    \[
        \eta_1' Y_1.\chi_1, \eta_2' Y_2.\chi_2, \dots, \eta_m' Y_m.\chi_m.
    \]
    An \emph{$\mathsf{II}$-signature} $r = \tuple{r(1), \dots, r(m)}$ is a sequence of $m$ ordinals.
    As with $\mathsf{I}$-signatures, denote by $r(k)$ the $k$th component of $r$ and write $r =_k r'$ iff the first $k$ components of $r$ are identical.
    Order the $\mathsf{II}$-signatures by the lexicographical order: $r<r'$ iff there is $k\in\{1,\dots, m\}$ such that $r =_k r'$ and $r(k+1)<r'(k+1)$.
    The lexicographical order is a well-ordering of the $\mathsf{II}$-signatures.

    As above, let $M_n$ be the model $M$ augmented with the correct valuations of the variables occurring in $\varphi$.
    Given a signature $r$, we define augmented models $M^r_0, \dots, M^r_n$ by
    \begin{align*}
        M_0^r     &:= V; \\
        M_{i+1}^r &:= \left\{
        \begin{array}{ll}
            M_i[Z_{i+1}\mapsto \|\eta_j' Y_j^{r_j}.\chi_j\|^{M_i^r}], & \text{if $ Z_{i+1} = Y_j$}; \\
            M_i[Z_{i+1}\mapsto \|\eta_{i+1} Z_{i+1}.\psi_{i+1}\|^{M_i}], & \text{if there is no $j$ such that $Z_{i+1} = Y_j$}.
        \end{array} \right.
    \end{align*}
    On $M_n^r$, the variables $Y_j$ owned by $\mathsf{II}$ are assigned their $r(j)$th approximant $\|\eta_j^{r(j)} Y_j.\chi_j\|$, and variables owned by $\mathsf{I}$ receive their correct value.
    If $M_n,v\models\psi$, we call $\tuple{v,\psi,\role}$ a true position; if $M_n,v \not\models\psi$, we call $\tuple{v,\psi,\role}$ a false position.
    Now, if $\tuple{v,\psi,\role}$ a true position, then there is a least signature $r$ such that $M_n^r,v\models\psi$.
    Similarly, if $\tuple{v,\psi,\role}$ a false position, then there is a least signature $r$ such that $M_n^r,v\not\models\psi$.
    Denote these signatures by $\mathsf{sig}^\mathsf{II}{v,\psi,\role}$.

    Similar to the first case, $\mathsf{I}$ cannot move in ways where the $\mathsf{II}$-signature increases, and we can build a strategy $\tau$ for $\mathsf{II}$ in a way that, eventually, their moves do not increase the signature.
    By the same argument as above, the strategy $\tau$ is winning.
\end{proof}

\section{The collapse to modal logic over \texorpdfstring{$\isf$}{IS5} models}
\label{sec::collapse}
In this section we use the game semantics to prove that the $\mu$-calculus collapses to modal logic over $\isf$, that is, that all $\mu$-formula is equivalent to a modal formula over $\isf$-models.
In the first subsection, we isolate the key lemma to this proof.
In the second subsection, we prove the collapse.

\subsection{The key lemma}
To prove the $\mu$-calculus' collapse over classical $\mathsf{S5}$-models, Alberucci and Facchini use the following result:
\begin{proposition}[Alberucci, Facchini \cite{alberucci2009modal}]
    Let $M=\tuple{M,\equiv,W}$ be an $\mathsf{S5}$-model.
    If $w\equiv v$, then $M,w\models\triangle\varphi$ iff $M,v\models\triangle\varphi$, where $\triangle \in \{\Box,\Diamond\}$.
\end{proposition}
We cannot prove the same result over $\isf$-models, but the following Lemma will suffice:
\begin{lemma}
    \label{lem::glue_lemma-intuitionistic}
    Let $M=\tuple{W,\emptyset,\preceq, \equiv, V}$ be an $\isf$-model.
    Let $\preceq;\equiv$ be the composition of $\preceq$ and $\equiv$.
    If $w\preceq;\equiv w'$, then
    \[
    M,w\models \triangle\varphi \text{ implies } M,w'\models \triangle\varphi,
    \]
  where $\triangle \in \{\Box,\Diamond\}$.
\end{lemma}
\begin{proof}
    Fix an $\isf$-model $M=\tuple{W,\emptyset,\preceq, \equiv, V}$.
    The composition $\preceq;\equiv$ is a transitive relation by Lemma \ref{lem::transitiveness-is5}.

    Also note that the worlds occurring in some position of a play are $\preceq;\equiv$-accessible from the previously occurring worlds.
    That is, when if players have gone through a position $\tuple{v,\psi,\role}$ and later $\tuple{v',\psi',\role'}$, then $v\preceq;\equiv v'$.
    This happens because $\preceq$ and $\equiv$ are reflexive relations and $\preceq;\equiv$ is transitive.

    Now, suppose $w\preceq;\equiv w'$ and $M,w\models\Diamond\varphi$.
    For all $v\succeq w$, there is $u\equiv v$ such that $M,u\models \varphi$.
    Let $v,v'$ be such that $w\preceq v\equiv w'\preceq v'$.
    By downward confluence, there is $u$ such that $v\preceq u \equiv v'$.
    By the transitivity of $\preceq$, $w\preceq u$.
    So there is $u'\succeq w$ such that $u \equiv u'$ and $M,u'\models \varphi$.
    As $v'\equiv u \equiv u'$, $v'\equiv u'$.
    So for all $v'\succeq w'$ there is $u'\equiv v'$ such that $M,u'\models\varphi$.
    That is, $M,w'\models \Diamond \varphi$.

    Similarly, suppose $w\preceq;\equiv w'$ and $M,w\models \Box\varphi$.
    Therefore $w\preceq;\equiv u$ implies $M,u\models \varphi$.
    Let $w'\preceq;\equiv u'$, then $w\preceq;\equiv u'$ by the transitiveness of $\preceq;\equiv$.
    So $M,u'\models\varphi$.
    Thus $M,w'\models\Box\varphi$.
\end{proof}

\subsection{The collapse}
We first show that the fixed-points for modal formulas can be reached in two steps.
Our proof is by contradiction.
This contradiction is not essential, but makes the proof easier to understand.
\begin{lemma}
    \label{lem::intuitionistic-eta_x-varphi_equiv_varphi2}
    Let $M=\tuple{W,\emptyset,\preceq, \equiv, V}$ be an $\isf$-model and $\varphi$ be a modal formula where $X$ is positive and appears only once in $\varphi$.
    Then
    \[
      \|\mu X.\varphi\|^M = \|\varphi(\varphi(\top))\|^M \text{ and }\|\nu X.\varphi\|^M = \|\varphi(\varphi(\bot))\|^M.
   \]
\end{lemma}
\begin{proof}
    We first show that $\|\nu X.\varphi\| = \|\varphi(\varphi(\top))\|$.
    Let $M=\tuple{W,W^\bot,\preceq,\equiv,V}$ be an $\isf$-models and $\nu X.\varphi$ be a well-named $\mu$-formula.
    We can also suppose that $\varphi$ is of the form $\alpha(\triangle \beta(X))$ with $\triangle \in \{\Box, \Diamond\}$.

    We show that $\nu X.\varphi$ is equivalent to $\varphi(\varphi(\top))$.
    As $X$ is positive in $\varphi(X)$, we have that $\|\varphi(\varphi(\varphi(\top)))\| \subseteq \|\varphi(\varphi(\top))\|$.
    So we need only to show that $\|\varphi(\varphi(\top))\| \subseteq \|\varphi(\varphi(\varphi(\top)))\|$.

    For a contradiction, suppose that $w\in \|\varphi(\varphi(\top))\|$ and $w\not\in \|\varphi(\varphi(\varphi(\top)))\|$.
    Then $\mathsf{I}$ has a winning strategy $\sigma$ for the evaluation game $\mathcal{G}_2 =\mathcal{G}(M,w\models\varphi(\varphi(\top)))$; and $\mathsf{II}$ has a winning strategy $\tau$ for the evaluation game $\mathcal{G}_3 =\mathcal{G}(M,w\models\varphi(\varphi(\varphi(\top))))$.
    We use $\sigma$ and $\tau$ to define strategies $\sigma'$ for $\mathsf{I}$ in $\mathcal{G}_3$ and $\tau'$ for $\mathsf{II}$ in $\mathcal{G}_2$.
    Remember that $\mathsf{I}$ starts on the role of $\verifier$ and $\mathsf{II}$ starts on the role of $\refuter$.

    We have the players use analogous strategies on both games.
    Suppose the players are in positions $\tuple{v,\psi(\top),\role}$ in $\mathcal{G}_2$ and $\tuple{v,\psi(\varphi(\top)),\role}$ in $\mathcal{G}_3$.
    Both positions have the same owner, in the same role.
    That is, if $\mathsf{I}$'s turn in some game, it is $\mathsf{I}$'s turn in both games; and the owner's role is $\verifier$ in some game, their role is $\verifier$ in both games.
    For example, suppose $\mathsf{I}$ is playing the role of $\refuter$ and the players are in positions $\tuple{v,\neg\psi(\top),\refuter}$ and $\tuple{v,\neg\psi(\varphi(\top)),\refuter}$ in $\mathcal{G}_2$ and $\mathcal{G}_3$.
    If $\mathsf{I}$ plays $\sigma(\tuple{v,\neg\psi(\top)\refuter}) = \tuple{v',\psi(\top)\verifier}$ in $\mathcal{G}_2$, they play $\tuple{v,\psi(\varphi(\top))\verifier}$ in $\mathcal{G}_3$.

    The players continue both games following the strategies described above until they get to a position of the form $\tuple{v,P}$ in both games; or they get to positions of the form $\tuple{w'',\triangle\beta(\top),\role}$ in $\mathcal{G}_2$ and $\tuple{w'',\triangle\beta(\varphi(\top)),\role}$ in $\mathcal{G}_3$.

    \emph{Case 1.} Suppose the players are in a position $\tuple{v,P,\role}$ in both games.
    Without loss of generality, suppose $\mathsf{I}$ is $\verifier$ and $\mathsf{II}$ is $\refuter$.
    As $\sigma$ is winning for $\mathsf{I}$ in $\mathcal{G}_2$, $v\in \|P\|$.
    As $\tau$ is winning for $\mathsf{II}$ in $\mathcal{G}_3$, $v\not\in \|P\|$.
    And so we have a contradiction.
    A similar contradiction is reached if $\mathsf{I}$ is $\refuter$ and $\mathsf{II}$ is $\verifier$.

    \emph{Case 2.} Suppose the players are in positions of the form $\tuple{w'',\triangle\beta(\top),\role}$ in $\mathcal{G}_2$ and $\tuple{w'',\triangle\beta(\varphi(\top)),\role}$ in $\mathcal{G}_3$.
    Without loss of generality, suppose $\mathsf{I}$ is $\verifier$ and $\mathsf{II}$ is $\refuter$.
    As $\tau$ is a winning strategy for $\mathsf{II}$ in $\mathcal{G}_3$, $w''\not \in \|\triangle\beta(\varphi(\top))\|$.
    Previously, the players must have been through some a position $\tuple{w',\triangle\beta(\varphi(\top)),\verifier}$ in $\mathcal{G}_2$.
    As $\sigma$ is a winning strategy for $\mathsf{I}$ in $\mathcal{G}_2$, $w'\in \|\triangle\beta(\varphi(\top))\|$.
    Note that, from the definition of the game, the reflexivity of $\preceq$ and $\equiv$, and the transitivity of $\preceq;\equiv$, we have that $w'\preceq;\equiv w''$.
    By Lemma \ref{lem::glue_lemma-intuitionistic}, $w''\in \|\triangle\beta(\varphi(\top))\|$ since $w'\in \|\triangle\beta(\varphi(\top))\|$.
    We have our contradiction.

    Either way, we conclude that $\|\varphi(\varphi(\top))\| \subseteq \|\varphi(\varphi(\varphi(\top)))\|$.
    And so $\|\nu X.\varphi\| = \|\varphi(\varphi(\bot))\|$.

    In classical semantics, we can prove $\|\mu X.\varphi\| = \|\varphi(\varphi(\bot))\|$ by a direct calculation.
    We cannot do the same in intuitionistic semantics as we cannot use the law of excluded middle.
    We have to prove it directly.

    First, $\|\varphi(\varphi(\bot))\| \subseteq \|\mu X.\varphi\|$ holds as $X$ is positive in $\varphi(X)$.
    If we suppose there is $w$ such that $w\in\|\mu X.\varphi\|$ and $w \not\in\|\varphi(\varphi(\bot))\|$, we get a similar contradiction.
\end{proof}

We are now able to show the constructive $\mu$-calculus' collapse to modal logic over $\isf$-models.
\begin{theorem}
    \label{thm::the_collapse}
    Over $\isf$-models, every $\mu$-formula is equivalent to a modal formula.
\end{theorem}
\begin{proof}
    We argue by structural induction on $\mu$-formulas.
    First, some of the easy cases.
    $P$ is equivalent to a modal formula, as it is a modal formula.
    Suppose the $\mu$-formulas $\varphi$ and $\psi$ are equivalent to modal formulas $\varphi'$ and $\psi'$, then $\varphi\land\psi$ is equivalent to $\varphi'\land\psi'$, $\varphi\to\psi$ is equivalent to $\varphi'\to\psi'$, $\Box\varphi$ is equivalent to $\Box\varphi'$, $\Diamond\varphi$ is equivalent to $\Diamond\varphi'$.

    Now, the interesting cases.
    As above, $\nu X.\varphi$ is equivalent to $\nu X.\varphi'$, where $\varphi'$ is a modal formula.
    By Lemma \ref{lem::intuitionistic-eta_x-varphi_equiv_varphi2}, $\nu X.\varphi'$ is equivalent to $\varphi'(\varphi'(\top))$, which is a modal formula.
    The same Lemma shows that $\mu X.\varphi'$ is equivalent to $\varphi'(\varphi'(\bot))$.

    Therefore every $\mu$-formula is equivalent to a modal formula over $\isf$-models.
\end{proof}

\section{The completeness of \texorpdfstring{$\muisf$}{muIS5}}
\label{sec::mu-is5}

In this section, we prove:
\begin{theorem}
    \label{thm::completeness-muisf}
    For all closed $\mu$-formula $\varphi$, $\muisf$ proves $\varphi$ iff $\varphi$ is true at all $\isf$-models.
\end{theorem}
We begin by proving the soundness of $\muisf$.
Then we show a formalized version of the collapse to modal logic.
At last, we use the provable collapse to prove the Truth Lemma for $\muisf$.
Our canonical model argument uses the notation of Balbiani \emph{et al.} \cite{balbiani2021constructive}, but the construction is similar to the canonical model for $\isf$ defined by Fischer Servi \cite{fischerservi1978finite}.

\subsection{Soundness}
The soundness of $\muisf$ is straightforward:
\begin{lemma}
    \label{lem::soundness}
    Fix a $\mu$-formula $\varphi$.
    Then $\varphi\in\mathsf{\mu IS5}$ implies $\varphi$ holds over all $\isf$-models.
\end{lemma}
\begin{proof}
    We will prove only the soundness of the axiom $\nu FP$ and rule $\mathbf{\nu Ind}$; the soundness of $\mu FP$ and rule $\mathbf{\mu Ind}$ are analogous.
    The soundness of these axioms and rules will follow from Lemma \ref{prop::monotoness_of_Gamma_varphi} and basic properties of monotone operators (see also \cite{arnold2001rudiments}).
    For the soundness of the axioms in $\isf$, see Fischer Servi \cite{fischerservi1978finite} and Ono \cite{ono1977intuitionistic}.

    Fix an $\isf$-model $M = \tuple{W,R,V}$.
    Suppose $w\models\nu X.\varphi$; that is, $w$ is in the greatest fixed-point $\|\nu X.\varphi\|$ of $\Gamma_{\varphi(X)}$.
    Therefore, $w\in \Gamma_{\varphi(X)}( \|\nu X.\varphi\| )$.
    And so $w\models\varphi(\nu X.\varphi)$.
    This implies $\nu FP$ is sound.
    Now, suppose $\psi\to\varphi(\psi)$ holds on every world in $W$.
    Then $\|\psi\| \subseteq \|\varphi(\psi)\|$.
    As $\|\nu X.\varphi\|$ is the greatest fixed-point of $\Gamma_{\varphi(X)}$, $\|\psi\| \subseteq \|\nu X.\varphi\|$ too.
    Therefore $\psi\to\nu X.\varphi$ holds in every world in $W$, and so $\mathbf{\nu Ind}$ is sound.
\end{proof}

\subsection{The provable collapse}
We now show that any $\mu$-formula is provably equivalent to a modal formula in $\muisf$.

We first prove a technical lemma showing that monotonicity for formulas without fixed-point operators is provable.
\begin{lemma}
    \label{lem::provable_monotonicity}
    Suppose $A\to B\in \mathsf{\mu IS5}$ and $\varphi(X)$ is a formula without fixed-point operators.
    If $X$ is positive in $\varphi(X)$, then $\varphi(A)\to\varphi(B)\in\mathsf{\mu IS5}$.
    If $X$ is negative in $\varphi(X)$, then $\varphi(B)\to\varphi(A)\in\mathsf{\mu IS5}$.
\end{lemma}
\begin{proof}
We prove this lemma using structural induction.
We prove only the cases where $X$ is positive, as the cases where $X$ is negative are similar.
\begin{itemize}
    \item For $\varphi$ is a proposition symbol or a variable symbol, then the result is immediate.

    \item Let $\varphi= \psi\lor \theta$.
    By the induction hypothesis, $\psi(A)\to\psi(B)\in\mathsf{\mu IS5}$ and $\theta(A)\to\theta(B)\in\mathsf{\mu IS5}$.
    As $[(\psi(A)\to\psi(B)) \land (\theta(A)\to\theta(B))] \to (\varphi(A)\to\varphi(B))$ is a tautology, $\varphi(A)\to\varphi(B)\in\mathsf{\mu IS5}$ follows by $\mathbf{MP}$.

    \item Let $\varphi= \psi\land \theta$.
    By the induction hypothesis, $\psi(A)\to\psi(B)\in\mathsf{\mu IS5}$ and $\theta(A)\to\theta(B)\in\mathsf{\mu IS5}$.
    As $[(\psi(A)\to\psi(B)) \land (\theta(A)\to\theta(B))] \to (\varphi(A)\to\varphi(B))$ is a tautology, $\varphi(A)\to\varphi(B)\in\mathsf{\mu IS5}$ follows by $\mathbf{MP}$.

    \item Let $\varphi = \Box\psi$.
    By the induction hypothesis, $\psi(A)\to \psi(B)\in\mathsf{\mu IS5}$.
    Then $\Box(\psi(A)\to \psi(B))\in\mathsf{\mu IS5}$ by $\mathbf{Nec}$, and so $\Box\psi(A)\to \Box\psi(B)\in\mathsf{\mu IS5}$ by $K$ and $\mathbf{MP}$.

    \item Let $\varphi = \Diamond\psi$.
    By the induction hypothesis, $\psi(A)\to \psi(B)\in\mathsf{\mu IS5}$.
    Then $\Box(\psi(A)\to \psi(B))\in\mathsf{\mu IS5}$ by $\mathbf{Nec}$, and so $\Diamond\psi(A)\to \Diamond\psi(B)\in\mathsf{\mu IS5}$ by $K$ and $\mathbf{MP}$.

    \item Let $\varphi = \neg\psi$.
    Then $X$ is negative in $\psi$ and so $\psi(B)\to\psi(A)\in\mathsf{\mu IS5}$ by the induction hypothesis.
    Since $(\psi(B)\to\psi(A)) \to (\neg\psi(A)\to\neg\psi(B))$ is a tautology, $\neg\psi(A)\to\neg\psi(B)\in\mathsf{\mu IS5}$ too.

    \item Let $\varphi = \psi\to\theta$.
    Then $X$ is negative in $\psi$ and positive in $\theta$.
    So $\psi(B)\to\psi(A)\in\mathsf{\mu IS5}$ and $\theta(A)\to\theta(B)\in\mathsf{\mu IS5}$ by the induction hypothesis.
    Since $[(\psi(B)\to\psi(A)) \land (\theta(A)\to\theta(B))] \to (\neg\psi(A)\to\neg\psi(B))$ is a tautology, $\varphi(A)\to\varphi(B)\in\mathsf{\mu IS5}$ too. \qedhere
\end{itemize}
\end{proof}

Now, we show that fixed-points of modal formulas are equivalent to modal formulas over $\muisf$.
This is a formal version of Lemma \ref{lem::intuitionistic-eta_x-varphi_equiv_varphi2}.
\begin{lemma}
    \label{lem::provable_collapse}
    If $\varphi$ has no fixed-point operators, then $\nu X.\varphi\leftrightarrow\varphi(\varphi(\top))\in\mathsf{\mu IS5}$ and $\mu X.\varphi\leftrightarrow\varphi(\varphi(\bot))\in\mathsf{\mu IS5}$.
\end{lemma}
\begin{proof}
    $\nu X.\varphi\to\top$ is in $\muisf$ as it is a tautology.
    By Lemma \ref{lem::provable_monotonicity}, we get that $\varphi(\nu X.\varphi) \to \varphi(\top)$.
    By $\mu FP$, $\nu X.\varphi\to\varphi(\nu X.\varphi)$ is in $\muisf$.
    By $\mathbf{MP}$, we have that $\nu X.\varphi \to \varphi(\top)$.
    By repeating this argument once, we get $\nu X.\varphi \to \varphi(\varphi(\top))$.

    Now, as $\varphi(\varphi(\top))\to \varphi(\varphi(\varphi(\top)))$ is valid on any $\isf$-model, $\varphi(\varphi(\top))\to \varphi(\varphi(\varphi(\top)))$ is provable $\muisf$ by the completeness of $\isf$.
    By $\mathbf{\nu Ind}$, $\varphi(\varphi(\top)) \to \nu X.\varphi$.

    The proof for $\mu X.\varphi$ is similar.
\end{proof}

Similar to how we proved Theorem \ref{thm::the_collapse}, we use  Lemma \ref{lem::provable_collapse} to prove the following theorem:
\begin{theorem}
    \label{thm::the_provable-collapse_for_CS5}
    Any $\mu$-formula is provably equivalent to a modal formula over $\muisf$.
    That is, for all $\mu$-formula $\varphi$, there is a modal formula $\varphi'$ such that $\varphi\leftrightarrow\varphi'$ is in $\muisf$.
\end{theorem}

\subsection{The canonical model}
We say $\Gamma$ is a \emph{$\muisf$-theory} if:
    $\Gamma$ is a set of formulas containing all the axioms of $\muisf$;
    $\Gamma$ is closed under under $\mathbf{MP}$;
    $\bot\not\in\Gamma$; and
    $\varphi\lor\psi\in \Gamma$ implies $\varphi\in\Gamma$ or $\psi\in\Gamma$.
Denote by $\Gamma^\Diamond$ the set $\{\varphi\mid \Diamond \varphi\in\Gamma\}$ and by $\Gamma^\Box$ the set $\{\varphi\mid \Box \varphi\in\Gamma\}$.
Define the \emph{canonical $\muisf$-model $M_c := \tuple{W_c, W^\bot_c, \preceq_c,\equiv_c, V_c}$} by:
\begin{itemize}
    \item $W_c := \{ \Gamma \mid \Gamma\text{ is a $\muisf$-theory} \}$;
    \item $W^\bot_c = \emptyset$;
    \item $\Gamma \preceq_c \Delta$ iff $\Gamma\subseteq\Delta$;
    \item $\Gamma \equiv_c \Delta$ iff $\Delta\subseteq \Gamma^\Diamond$ and $\Gamma^\Box \subseteq \Delta$; and
    \item $\Gamma\in V_c(\varphi)$ iff $P\in\Gamma$.
\end{itemize}

\begin{lemma}
    \label{lem::equiv_is_equiv}
    Let $M_c := \tuple{W_c, W^\bot_c, \preceq_c,\equiv_c, V_c}$ be the canonical $\muisf$-model.
    The relation $\equiv_c$ is an equivalence relation.
\end{lemma}
\begin{proof}
    By $T$ and $\mathbf{MP}$, $\varphi\in\Gamma$ implies $\Diamond\varphi\in\Gamma$ and $\Box\varphi\in\Gamma$ implies $\varphi\in\Gamma$.
    So $\Gamma\subseteq \Gamma^\Diamond$ and $\Gamma^\Box\subseteq\Gamma$. Thus $\Gamma\equiv_c\Gamma$.

    Let $\Gamma \equiv_c \Delta \equiv_c \Sigma$.
    Then $\Delta \subseteq \Gamma^\Diamond$, $\Gamma^\Box\subseteq\Delta$, $\Sigma\subseteq\Delta^\Diamond$, and $\Delta^\Box\subseteq\Sigma$.
    Suppose $\varphi\in \Gamma^\Box$, then $\Box\varphi\in\Gamma$ and $\Box\Box\varphi\in\Gamma$, so $\Box\varphi\in\Gamma^\Box\subseteq \Delta$.
    Thus $\varphi\in\Sigma$.
    Suppose $\varphi\in\Sigma$, then $\varphi\in\Delta^\Diamond$ and $\Diamond\varphi\in\Delta\subseteq \Gamma^\Diamond$.
    So $\Diamond\Diamond\varphi\in\Gamma$.
    By $4$ and $\mathbf{MP}$, $\Diamond\varphi\in\Gamma$ and so $\varphi\in\Gamma^\Diamond$.
    Therefore $\Gamma\equiv_c\Sigma$.

    Let $\Gamma \equiv_c\Delta$, then $\Delta \subseteq \Gamma^\Diamond$ and $\Gamma^\Box\subseteq\Delta$.
    We want to show $\Gamma \subseteq \Delta^\Diamond$, $\Delta^\Box\subseteq\Gamma$.
    Let $\varphi\in\Gamma$.
    By $T$ and $\mathbf{MP}$, $\Diamond\varphi\in\Gamma$.
    By $5$, $\Box\Diamond\varphi\in\Gamma$, so $\Diamond\varphi\in\Gamma^\Box\subseteq\Delta$.
    Thus $\varphi\in\Delta^\Diamond$.
    Now, suppose $\varphi\in\Delta^\Box$.
    So $\Box\varphi\in\Delta\subseteq \Gamma^\diamond$.
    Thus $\Diamond\Box\varphi\in\Gamma$.
    By $5$, $\Box\varphi\in \Gamma$ and so $\varphi\in\Gamma$.
    Therefore $\Delta \equiv_c \Gamma$.
\end{proof}

\begin{lemma}
    \label{lem::equiv_is_confluent}
    Let $M_c := \tuple{W_c, W^\bot_c, \preceq_c,\equiv_c, V_c}$ be the canonical $\muisf$-model.
    The relation $\equiv_c$ is backward confluent.
    That is, if $\Gamma,\Delta,\Sigma\in W_c$ and $\Gamma\equiv_c\Delta\preceq_c\Sigma$, then there is $\Phi\in W_c$ such that $\Gamma\preceq_c\Phi\equiv_c\Sigma$
\end{lemma}
\begin{proof}
    Suppose $\Gamma\equiv_c\Delta\preceq_c\Sigma$.
    By hypothesis, $\Delta\subseteq \Gamma^\Diamond$, $\Gamma^\Box\subseteq\Delta$, and $\Delta\subseteq \Sigma$.

    Let $\Upsilon$ be the closure of $\Gamma \cup \{\Diamond\varphi \mid \varphi\in\Sigma\}$ under $\mathbf{MP}$.
    We first show that, if $\Box\varphi$ is provable formulas in $\Upsilon$, then $\varphi\in\Sigma$.
    There are $\{\psi_i\}_{i<m}\subseteq \Gamma$ and $\{\chi_i\}_{i<n}\subseteq \Sigma$ such that $\isf$ proves $\bigwedge_{j<n}\Diamond\chi_j \land \bigwedge_{i<m}\psi_i \to \Box \varphi$. By $\mathbf{Nec}$ and $K$, $\isf$ proves $\bigwedge_{j<n}\Box\Diamond\chi_j \to \Box( \bigwedge_{i<m}\psi_i \to \Box \varphi)$ and $\bigwedge_{j<n}\Box\Diamond\chi_j \to (\Diamond\bigwedge_{i<m}\psi_i \to \Diamond\Box \varphi)$.
    Since each $\chi_j$ is in $\Sigma$, so is $\Box\Diamond\chi_j$, by $T$ and $5$ along with $\mathbf{MP}$.
    Since $\bigwedge_{i<m}\psi_i\in\Gamma$, $\Diamond\bigwedge_{i<m}\psi_i\in\Delta$, and thus $\Diamond\bigwedge_{i<m}\psi_i\in\Delta$ too.
    By repeated applications of $\mathbf{MP}$, we have $\Diamond\Box\varphi\in\Sigma$.
    By $5$ and $T$, we have $\varphi\in\Sigma$.

    By an application of Zorn's Lemma, there is a maximal set $\Phi$ such that: $\Phi$ is a consistent set of formulas containing $\Upsilon$; $\Phi$ closed under $\mathbf{MP}$; and $\Box\varphi\in\Phi$ implies $\varphi\in\Sigma$.
    Suppose $\varphi\lor\psi\in\Phi$.
    By $T$ and $5$, $\Box\Diamond(\varphi\lor\psi)\in\Phi$, and so $\Diamond(\varphi\lor\psi)$.
    Since $\isf\vdash DP$ and $\Sigma$ is a theory, at least one of $\Diamond\varphi$ and $\Diamond\psi$ is in $\Sigma$.
    Suppose both $\Phi\cup\{\varphi\}$ and $\Phi\cup\{\psi\}$ are inconsistent.
    Then $\bot\in\Phi$, thus $\Box\Diamond\bot\in\Phi$, and so $\Diamond\bot\in\Sigma$; this is a contradiction.
    So at least one of $\Phi\cup\{\varphi\}$ and $\Phi\cup\{\psi\}$ is consistent.

    If $\Phi\cup\{\varphi\}$ is consistent and $\Diamond\varphi\in\Sigma$, we can show by the same argument as the paragraph above that, if $\Box\chi$ is a consequence of $\Phi\cup\{\varphi\}$, then $\chi$ is a consequence of $\Sigma$.
    Therefore the closure of $\Phi\cup\{\varphi\}$ under $\mathbf{MP}$ is a subset of $\Phi$, so $\varphi\in\Phi$ from the beginning.
    Similarly, if $\Phi\cup\{\psi\}$ is consistent and $\Diamond\psi\in\Sigma$, then $\psi\in\Phi$.
    If $\Phi\cup\{\varphi\}$ is inconsistent and $\Diamond\psi\not\in\Sigma$, we get that $\Diamond\bot\in\Sigma$; and so $\bot\in\Sigma$ by $N$, a contradiction.
    $\Phi\cup\{\psi\}$ is inconsistent and $\Diamond\varphi\not\in\Sigma$ give a similar contradiction.
    Therefore either $\varphi\in\Phi$ or $\psi\in\Phi$.
    Therefore $\Phi$ is a $\isf$-theory.

    Trivially, $\Gamma\subseteq\Phi$ and so $\Gamma\preceq_c\Phi$.
    If $\Box\varphi\in\Phi$, then $\varphi\in\Sigma$ by the construction of $\Phi$.
    If $\varphi\in\Sigma$ then $\Diamond\varphi\in\Upsilon\subseteq\Phi$.
    Therefore $\Phi\equiv_c\Sigma$.
    This concludes the lemma.
\end{proof}

We now have have:
\begin{lemma}
    \label{lem::mc-is-canonical}
    The canonical $\muisf$-model $M_c$ is an $\isf$-model.
\end{lemma}
\begin{proof}
    Since the subset relation $\subseteq$ is a preorder, $\preceq_c$ is a reflexive and transitive relation over $W_c$.
    $W^\bot_c$ is empty by definition.
    The relation $\equiv_c$ is an equivalence relation over $W_c$ by Lemma \ref{lem::equiv_is_equiv}
    $M_c$ also satisfies the convergence requirements by Lemma \ref{lem::equiv_is_confluent} and Proposition \ref{prop::confluences-are-equivalent}.
    It follows from the definition that $\preceq_c$ preserves the truth of propositions.
\end{proof}

With the provable collapse over $\muisf$, we can prove Truth Lemma for the canonical $\muisf$-model.
\begin{lemma}
    \label{lem::truth-lemma}
    Let $M_c := \tuple{W_c, W^\bot_c, \preceq_c,\equiv_c, V_c}$ be the canonical $\muisf$-model.
    For $\muisf$-theory $\Gamma$ and all closed $\mu$-formula $\varphi$,
    \[
        M_c,\Gamma\models \varphi \text{ iff } \varphi\in\Gamma.
    \]
\end{lemma}
\begin{proof}
    The proof is by structural induction on modal formulas.
    \begin{itemize}
        \item If $\varphi = P$, then the lemma holds by the definition of $M_c$.

        \item If $\varphi = \bot$, then the lemma holds by the definition of the semantics and of $W_c^\bot$.

        \item If $\varphi = \psi_1\land\psi_2$, then
        \begin{align*}
            \Gamma\models \psi_1\land\psi_2 &\text{ iff } \Gamma\models\psi_1 \text{ and } \Gamma \models \psi_2 \\
            &\text{ iff } \psi_1\in\Gamma \text{ and }\psi_2\in\Gamma \\
            &\text{ iff } \psi_1\land \psi_2\in\Gamma.
        \end{align*}

        \item If $\varphi = \psi_1\lor\psi_2$, then
        \begin{align*}
            \Gamma\models \psi_1\lor\psi_2 &\text{ iff } \Gamma\models\psi_1 \text{ or } \Gamma \models \psi_2 \\
            &\text{ iff } \psi_1\in\Gamma \text{ or }\psi_2\in\Gamma \\
            &\text{ iff } \psi_1\lor \psi_2\in\Gamma.
        \end{align*}
        Here we use that if $\psi_1\lor\psi_2\in\Gamma$ then $\psi_1\in\Gamma$ or $\psi_2\in\Gamma$, as $\Gamma$ is a $\muisf$ theory.

        \item Let $\varphi := \psi_1\to\psi_2$.

        First suppose that $\psi_1\to\psi_2\in\Gamma$.
        Let $\Delta$ be a theory such that $\Gamma\preceq_c\Delta\models\psi_1$.
        By the induction hypothesis, $\psi_1\in\Delta$.
        As $\Gamma\preceq_c\Delta$, $\psi_1\to\psi_2\in\Delta$.
        By $\mathbf{MP}$, $\psi_2\in\Delta$.
        So $\Gamma\models\psi_1\to\psi_2$.

        Now suppose that $\psi_1\to\psi_2\not\in\Gamma$.
        Take $\Upsilon$ to be the closure of $\Gamma\cup \{\psi_1\}$ under the derivation rules.
        If $\psi_2\in\Upsilon$, then there is $\chi\in \Gamma$ such that $(\chi \land \psi_1)\to \psi_2\in \muisf$.
        And so $\chi \to (\psi_1\to \psi_2)\in \muisf$.
        As $\chi\in\Gamma$, this means $\psi_1\to\psi_2\in \Gamma$, a contradiction.
        Therefore $\psi_2\not\in\Upsilon$.
        By Zorn's Lemma, we can build a theory $\Sigma$ which contains $\Upsilon$ and do not prove $\psi_2$.
        By the induction hypothesis, $\Sigma\models \psi_1$ and $\Sigma\not\models\psi_2$.
        As $\Gamma\preceq_c\Sigma$, $\Gamma\not\models\psi_1\to\psi_2$.

        \item Let $\varphi = \neg\psi$. This case follows by the equivalence between $\neg\psi$ and $\varphi\to\bot$ over intuitionistic logic.

        \item Let $\varphi = \Box\psi$.

        First suppose that $\Box\psi\in\Gamma$.
        Let $\Gamma\preceq_c\Delta\equiv_c\Sigma$.
        Then $\Box\psi\in\Delta$ and $\psi\in\Sigma$.
        By induction hypothesis, $\Sigma\models\psi$.
        So $\Gamma\models\Box\psi$.

        Now suppose that $\Box\psi\not\in\Gamma$.
        Define $\Sigma:= \Gamma^\Box$.
        By definition, $\psi\not\in\Sigma$.
        By the induction hypothesis, $\Sigma\not\models\psi$.
        Now we show that $\Gamma\equiv_c\Sigma$.
        $\Gamma^\Box\subseteq\Sigma$ follows by definition.
        Let $\theta\in\Sigma$.
        Then $\Box\theta\in\Gamma$.
        By two applications of $T$, $\Diamond\theta\in\Gamma$.
        So $\theta\in\Gamma^\Diamond$.
        So $\Gamma\equiv_c\Sigma$.
        Therefore $\Gamma\preceq_c\Gamma\equiv_c\Sigma\not\models\psi$, and thus $\Gamma\not\models\Box\psi$.

        \item Let $\varphi = \Diamond\psi$.

        First suppose that $\Diamond\psi\in\Gamma$.
        Let $\Delta$ be a theory such that $\Gamma\preceq_c\Delta$.
        Furthermore, suppose $\Delta$ is consistent.
        Let $\Upsilon$ be the closure under derivation rules of $\Delta^\Box\cup\{\psi\}$.
        $\Upsilon^\Box\subseteq\Sigma$ holds by definition.
        Let $\theta\in\Sigma$, then $\chi\land \psi \to \theta\in \muisf$ for some $\chi\in\Upsilon^\Box$.
        Thus $\chi\to (\psi\to\theta)\in \muisf$ and $\Box\chi\to \Box(\psi\to\theta)\in \muisf$.
        So $\Box(\psi\to\theta)\in \Upsilon$.
        By $K$, $\Diamond\psi\to\Diamond\theta\in\Upsilon$.
        So $\Diamond\theta\in\Upsilon$.
        By Zorn's Lemma, there is a theory $\Sigma$ containing $\Upsilon$ such that $\Sigma\equiv_c\Delta$.
        By induction hypothesis, $\Sigma\models\psi$.
        Therefore $\Gamma\models\Diamond\psi$.

        Now suppose that $\Diamond\psi\not\in\Gamma$.
        Let $\Delta$ be such that $\Gamma\equiv_c\Delta$ and $\psi\in\Delta$.
        By the definition of $\equiv_c$, $\Delta\subseteq \Gamma^\Diamond$, so $\psi\in\Gamma^\Diamond$.
        Therefore $\Diamond\psi\in\Gamma$, a contradiction.
        We conclude that for all $\Delta$, if $\Gamma\preceq_c\Gamma\equiv_c\Delta$, then $\psi\not\in\Delta$.
        By the induction hypothesis, $\Delta\not\models\psi$.
        Therefore $\Gamma\not\models\Diamond\psi$.

        \item Let $\varphi$ be $\nu X.\psi(X)$.
        We want to show that $M_c,\Gamma\models \nu X.\psi$ iff $\nu X.\psi \in \Gamma$.

        By Lemma \ref{lem::provable_collapse}, $\nu X.\psi$ is provably equivalent to some modal formula $\varphi'$.
        So $\varphi\leftrightarrow\varphi' \in \Gamma$.
        Thus:
        \[
          \nu X.\psi\in\Gamma \iff \varphi'\in \Gamma \iff M_c,\Gamma\models \varphi' \iff M_c,\Gamma\models \nu X.\psi.
        \]
        The first equivalence holds by $\mathbf{MP}$, the second by completeness for $\isf$, and the last from the soundness of $\muisf$.

        \item Let $\varphi$ be $\nu X.\psi(X)$.
        By a proof similar to the paragraph above, we prove that $M_c,\Gamma\models \mu X.\psi$ iff $\mu X.\psi \in \Gamma$.
    \end{itemize}

    This finishes the proof of Lemma \ref{lem::truth-lemma}.
\end{proof}

\begin{proof}[Proof of Theorem \ref{thm::completeness-muisf}]
    Let $\varphi$ be a closed $\mu$-formula.
    If $\muisf$ proves $\varphi$, then $\varphi$ is true at all $\isf$-models by Lemma \ref{lem::soundness}.
    Now, suppose $\muisf$ does not prove $\varphi$.
    By Zorn's Lemma, there is an $\muisf$-theory $\Gamma$ such that $\varphi\not\in\Gamma$.
    Therefore, $\varphi$ does not hold over $\Gamma$ in the canonical model by Lemma \ref{lem::mc-is-canonical}; and so $\varphi$ is not true in all $\isf$-models.
\end{proof}

\section{Future Work}
\label{sec::future-work}
We now present some topics for research work that we are currently working on.
Most of these are centered on non-classical variants of $\mathsf{S4}$ with fixed-points.

\paragraph{Semantics for the constructive modal cube}
On one hand, the logics in the modal cube were studied by Arisaka \emph{et al.} \cite{arisaka2015nested}, who defined and proved the correctness of a nested sequent calculus them.
On the other hand, the bi-relational Kripke semantics have not been studied in general.
The author is only aware of such semantics for the constructive modal logic $\ck$ \cite{mendler2005constructive} and the constructive variant $\mathsf{CS4}$ of the modal logic $\mathsf{S4}$ \cite{alechina2001categorical,balbiani2021constructive}.
In particular, bi-relational semantics of the constructive variant $\mathsf{CS5}$ of $\mathsf{S5}$ has not been studied yet.

\paragraph{Completeness results}
The collapse over non-classical variants of $\mathsf{S5}$ greatly simplifies the proof of the completeness of the logic $\muisf$.
This collapse is not available for variants of $\mathsf{S4}$.
We conjecture that the completeness of non-classical variants of $\mathsf{S4}$ may be proved by combining the methods in Balbiani \emph{et al.} \cite{balbiani2021constructive} and Baltag \emph{et al.} \cite{baltag2023topological}.
The collapse is also not useful in the multimodal case: over multimodal $\mathsf{S5}$, the $\mu$-calculus' alternation hierarchy is strict \cite{pacheco2024fusions}.
This leaves the completeness of multimodal non-classical variants of $\mathsf{S5}$ open.
For an overview of multimodal logics, see \cite{kurucz2007Combining,carnielli2020Combining}.

\paragraph{Alternation hierarchy}
The Kripke models and $\mu$-formulas witnessing the $\mu$-calculus' alternation hierarchy strictness are also constructive Kripke models and constructive $\mu$-formulas.
Therefore the constructive $\mu$-calculus' alternation hierarchy is also strict.
Similarly, the $\mu$-calculus does not collapse to modal logic over non-classical variants of $\mathsf{S4}$.
We plan to investigate whether it collapses to its alternation-free hierarchy in a future paper.
Do note that the proofs of the collapse to the alternation-free hierarchy over $\mathsf{S4}$ found in \cite{alberucci2009modal} and \cite{dagostino2010transitive} do not translate to the constructive $\mu$-calculus.

\paragraph{Decidability results}
The decidability of constructive and intuitionistic variants of $\mathsf{S4}$ were settled recently by Balbiani \emph{et al.} \cite{balbiani2021constructive} and Girlando \emph{et al.} \cite{girlando2023intuitionistic}, respectively.
It is also known that the extension of $\mathsf{S4}$ by fixed-points is decidable, by the work of Baltag \emph{et al.} \cite{baltag2023topological}.
These results suggest that constructive and intuitionistic variants of $\mathsf{S4}$ with fixed-point operators might also be decidable.


\bibliographystyle{alphaurl}
\bibliography{intuitionistic-mu-calculus}

\end{document}